\documentclass[12pt]{amsart}

\usepackage{tikz}
\usepackage{color}
\usepackage{amsmath, amsthm, amssymb}
\usepackage{amsfonts}
\usepackage[ansinew]{inputenc}
\usepackage[dvips]{epsfig}
\usepackage{graphicx}
\usepackage[english]{babel}
\theoremstyle{plain}
\newtheorem{thm}{Theorem}[section]
\newtheorem{cor}[thm]{Corollary}
\newtheorem{lem}[thm]{Lemma}
\newtheorem{prop}[thm]{Proposition}

\theoremstyle{definition}
\newtheorem{defi}[thm]{Definition}

\theoremstyle{remark}
\newtheorem{rem}[thm]{Remark}

\numberwithin{equation}{section}

\renewcommand{\epsilon}{\varepsilon}

\renewcommand{\leq}{\leqslant}
\renewcommand{\le}{\leqslant}
\renewcommand{\geq}{\geqslant}
\renewcommand{\ge}{\geqslant}

\newcommand{\average}{{\mathchoice {\kern1ex\vcenter{\hrule height.4pt
width 6pt depth0pt} \kern-9.7pt} {\kern1ex\vcenter{\hrule
height.4pt width 4.3pt depth0pt} \kern-7pt} {} {} }}

\newcommand\restrict[1]{\raisebox{-.5ex}{$|$}_{#1}}

\def\R{\mathbb{R}}
\def\cC{\mathcal{C}}
\def\N{\mathbb{N}}

\begin{document}
\author{Eleonora Cinti, Joaquim Serra, and Enrico Valdinoci}

\address{E.C, J.S, E.V, Wierstrass Institute for Applied Analysis and Stochastics,
Mohrenstr. 39,
10117 Berlin (Germany) }

\email{cinti@wias-berlin.de, serra@wias-berlin.de, valdinoc@wias-berlin.de}

\title[Flatness and $BV$-estimates]{Quantitative flatness results and $BV$-estimates \\ for stable nonlocal minimal surfaces}
%


\begin{abstract}
We establish quantitative properties of minimizers and stable sets for nonlocal interaction functionals, including the $s$-fractional perimeter as a particular case. 

On the one hand, we establish universal $BV$-estimates in every dimension $n\ge 2$ for stable sets. Namely, we prove that any stable set in $B_1$ has finite classical perimeter  in $B_{1/2}$, with a universal bound. This nonlocal result is new  even in the case of $s$-perimeters and its local counterpart (for classical stable minimal surfaces) was known only for simply connected two-dimensional surfaces immersed in $\R^3$.

On the other hand, we prove quantitative flatness estimates for minimizers and stable sets in low dimensions $n=2,3$. More precisely, we show that a stable set in $B_R$, with $R$ large, is very close in measure to being a half space in $B_1$ ---with a quantitative estimate on the measure of the symmetric difference. As a byproduct, we obtain new classification results for stable sets in the whole plane.
\end{abstract}


\maketitle

\tableofcontents

\section{Introduction}\label{intro}

In this paper we establish quantitative properties of minimizers and stable sets of nonlocal interaction functionals of perimeter type. 
We consider very general ---possibly anisotropic and not scaling invariant functionals--- including, as particular cases, the fractional $s$-perimeter and its anisotropic version, introduced respectively in \cite{CRS} and \cite{L}.

The results that we obtain can be
grouped, roughly speaking, into the following categories:
\begin{itemize}
\item Local $BV$-estimates (universal bounds for the classical perimeter) and sharp energy estimates for minimizers and stable sets,
\item Existence results and compactness of minimizers,
\item Quantitative flatness results.

\end{itemize}
Before giving the most general statements of the results in the paper,
we just state them for the case of fractional $s$-perimeter. Even in this very particular case, the results are new and interesting
in themselves. 

The precise setting of the (most general)
nonlocal perimeter functionals that we consider will be discussed in
Subsection~\ref{setting}. In particular,
in the forthcoming Definitions~\ref{minim} and~\ref{stable}
we precise the notions of {\em minimizers} and {\em stable sets}.
Our results are stated in their full generality later on in Subsection~\ref{RIS:S}
---after having given in Subsection~\ref{R:W:M}
several concrete motivations for the problems under consideration.

We next state, in the case of the $s$-perimeter,  our main $BV$-estimate.
This result is a particular case of our Theorem \ref{thmstable}.  
It gives a universal bound on the classical perimeter in $B_{1/2}$ of any stable minimal set in $B_1$. 
As said above, the precise notion of stable solution will be given in Definition~\ref{stable}, and it is an appropriate weak formulation of the nonnegativity of the second variation of the functional.

\begin{thm}\label{89:TH}
Let $s\in(0,1)$, $R>0$ and $E$ be a stable set in the ball~$B_{2R}$
for the nonlocal $s$-perimeter functional. Then, the classical perimeter of~$E$
in~$B_R$
is bounded by~$CR^{n-1}$, where $C$ depends only on $n$ and $s$.

Moreover, the $s$-perimeter of~$E$ in~$B_R$
is bounded by~$CR^{n-s}$.
\end{thm}
Moreover, as a consequence of Theorem \ref{thmstable}, we establish the same result for the anisotropic fractional perimeter considered in \cite{L}.
%

To better appreciate Theorem \ref{89:TH} let us compare it with the best known similar results for classical minimal surfaces. 
A universal perimeter estimate for (local) stable minimal surfaces is only known for the case of {\em two-dimensional} stable minimal surfaces that are {\em simply connected} and {\em immersed}  in $\R^3$.
Conversely, the perimeter estimate in our Theorem \ref{89:TH} holds in every dimension and without topological constraints.
The perimeter estimate for the classical case is a result due to Pogorelov \cite{P}, and Colding and Minicozzi \cite{CM} ---see also \cite[Theorem 2]{Meeks} and \cite[Lemma 34]{White}, it reads as follows

\begin{thm}[\cite{P,CM}]\label{classical}
Let $D$ be a simply connected, immersed,  stable minimal disk of geodesic radius $r_0$ on a minimal (two-dimensional) surface $\Sigma\subset \R^3$, then
\[ \pi r_0^2 \le {\rm Area}\,(D)\le \frac{4}{3} \pi r_0^2. \]
\end{thm}

As said above, our estimate for nonlocal perimeters is stronger in the sense that we do not need $\partial E$ to be simply connected and immersed.
In fact, an estimate exactly like ours can not hold for classical stable minimal surfaces  since a large number of parallel planes is always a classical stable minimal surface with arbitrarily large perimeter in $B_1$.

The proof of Theorem \ref{classical} uses crucially the fact that for two-dimensional minimal surfaces the sum of the squares of the principal curvatures $\kappa_1^2+\kappa_2^2$ equals $2\kappa_1\kappa_2=-2K$, where $K$ is the Gau{\ss} curvature ---since on a minimal surface   $\kappa_1+\kappa_2=0$. Then, the stability inequality reads as $\int_D |\nabla\xi|^2 + 2K\xi^2 \ge 0$. By plugging a suitable radial test function $\xi$ in this stability inequality, using  the Gau{\ss}-Bonnet formula to relate $\int_{D_r}K$ and $\frac{d}{dr}{\rm Length}\,(\partial D_r)$, and integrating by parts in the radial variable, one proves the bound ${\rm Area}\,(D)\le \frac{4}{3} \pi r_0^2$.
This elegant proof is unfortunately quite rigid and only applies to two-dimensional surfaces.

Having a universal bound for the classical perimeter of embedded minimal surfaces in every dimension $n\ge4$ would be a decisive step towards proving the following well-known and long standing conjecture: 
{\em The only  stable embedded minimal (hyper)surfaces in  $\R^n$ are hyperplanes as long as the dimension of the ambient space is less than or equal to 7}.
Indeed, it would open the door to use the monotonicity formula to prove that  blow-downs of stable surfaces are stable minimal cones ---which are completely classified.
On the other hand, without a universal perimeter bound, the sequence of blow-downs could have perimeters converging to $\infty$.
In the same direction, we believe that our result in Theorem \ref{89:TH} can be used to reduce the classification of stable $s$-minimal surfaces in the whole $\R^n$ to the classification of stable cones

We note that our nonlocal estimate gives a control on the classical perimeter (i.e. the $BV$-norm of the characteristic function), which is stronger ---both from the geometric and functional space perspective--- than a control on the $s$-perimeter (i.e. on the $W^{s,1}$ norm of the characteristic function). The sharp $s$-perimeter estimate stated in Theorem \ref{89:TH} is obtained  as a consequence of the estimate for the classical perimeter using a standard interpolation.

Since it is well-known \cite{BRE, Davila, CV0, ADM} that the classical perimeter is the limit as $s\uparrow 1$ of the  nonlocal $s$-perimeter (suitably renormalized), it is natural to ask whether our results give some informations in the limit case $s=1$. Unfortunately, our proof relies strongly on the nonlocal character of the $s$-perimeter and the constant $C$ appearing in Theorem \ref{89:TH} blows up as $s\uparrow 1$.

The more general forms of our $BV$-estimates have quite remarkable consequences regarding the existence and 
compactness of minimizers ---see Theorem \ref{existence} and Lemma \ref{compactness}.
These existence and compactness results are nontrivial since they apply in particular to some perimeter functionals that are finite on every measurable set.
Thus, although all the perimeter functionals that we consider are lower semicontinuous, 
sequences of sets of finite perimeter are in principle not compact in $L^1$.
Thanks to our $BV$-estimates, we can obtain robust compactness results that serve to prove existence of minimizers in a very general framework.

We next give our quantitative flatness estimate in dimension $n=2$ for the case of the $s$-perimeter.  This result is a particular case of our Theorem \ref{aniso}.
It states that stable sets in a large ball $B_R$ are close to being a halfplane in $B_1$, with a quantitative control on the measure of the 
symmetric difference that decays to $0$ as $R\to \infty$.

\begin{thm}\label{T:EX2}
Let the dimension of the ambient space be equal to~$2$.
Let~$R\ge2$ and~$E$ be a stable  set in the ball~$B_{R}$
for the $s$-perimeter.

Then, there exists a halflplane~$\mathfrak h$ such that $|(E\triangle \mathfrak h)\cap B_1| \le CR^{-s/2}$.

Moreover, after a rotation, we have that~$E\cap B_1$ is
the graph of a measurable function~$g:(-1,1)\to(-1,1)$ with ${\rm osc\,}g\le CR^{-s/2}$
outside a ``bad'' set~${\mathcal{B}} \subset (-1,1)$ with measure
 $C R^{-s/2}$.
\end{thm}

The previous result provides a quantitative version of the classification result in \cite{SV} which says that if $E$ is a  minimizer of the $s$-perimeter in any compact set of $\R^2$, then it is necessarily a halfplane.
Moreover, Theorem~\ref{T:EX2} extends this classification result to the class of stable sets.

In Corollary ~\ref{thm1dim3} we will obtain also results in dimension $n=3$  for minimizers of anisotropic interactions with a finite range of dependence (i.e. for ``truncated kernels'').

The proofs of our main results have, as starting point, a nontrivial refinement of the variational argument introduced by Savin and one of the authors in \cite{SV, SV-mon} to prove that halfplanes are the only cones minimizing the $s$-fractional perimeter in every compact set of $\R^2$. Namely, we consider perturbations $E_{R,t}$ of a minimizer $E$ which coincide with $E$ outside $B_R$ and are translations $E+t\boldsymbol v $  of $E$ in $B_{R/2}$ ---with ``infinitesimal'' $t>0$. A first step in the proof is estimating how much $P_{K,B_R}(E_{R,t})$ differs from $P_{K,B_R}(E)$ depending on $R$ ---this is done in Lemma \ref{lem2A}. By exploiting the nonlocality of the perimeter functional, the previous control on $ P_{K,B_R}(E_{R,t})-P_{K,B_R}(E)$ is translated into a control on the minimum between $|E_{R,t}\setminus E|$ and $|E\setminus E_{R,t}|$ ---the crucial estimates for this are given in Lemmas \ref{lemEFdelta} and \ref{lemEFtdelta}. Then, a careful geometric analysis allows us to deduce our main results ---i.e. Theorems \ref{thmstable}, \ref{BV-est}, \ref{aniso}, \ref{flatness} and their corollaries.
We emphasize that we always use \emph{arbitrarily small} perturbations of our set $E$. That is why we can establish some results for \emph{stable} sets.


In the following subsections, we introduce the mathematical
framework of nonlocal perimeters, we discuss some motivations
for this general framework, 
and we present the main results obtained.

\subsection{The mathematical framework of nonlocal perimeter functionals}\label{setting}

The notion of fractional perimeter was introduced in \cite{CRS}. Let $s\in (0,1)$. Given a bounded domain $\Omega\subset \R^n$, we define the fractional $s$-perimeter of a measurable set $E\subset \R^n$ relative to $\Omega$ as
\begin{equation}\label{def-nlps} P_{s,\Omega}(E):=L_s(E\cap \Omega, \mathcal C E\cap \Omega) + L_s(E\cap \Omega, \mathcal C E\setminus \Omega) + L_s(E\setminus \Omega,  \mathcal C E\cap \Omega),
\end{equation}
where $\mathcal C E$ denotes the complement of $E$ in $\R^n$ and the interaction $L_s$ of two disjoint measurable sets $A, B$ is defined by
\[ L_s(A,B) := \int_A\int_B \frac{dx\,d\bar x}{|x-\bar x|^{n+s}}.\]
Roughly speaking, this $s$-perimeter captures the interactions
between a set~$E$ and its complement. These interactions
occur in the whole of the space and are weighted by
a (homogeneous and rotationally invariant) kernel with
polynomial decay (see Figure 1). Here, the role of the domain~$\Omega$
is to ``select'' the contributions which arise in a given
portion of the space and to ``remove'' possible infinite
contributions to the energy which come from infinity but which
do not change the variational problem.

\begin{figure}[h]
	\centering	
	\resizebox{15cm}{!}{
	
		\begin{tikzpicture}[scale=1]
		
			\node (myfirstpic) at (0,0) {\includegraphics[scale=0.6]{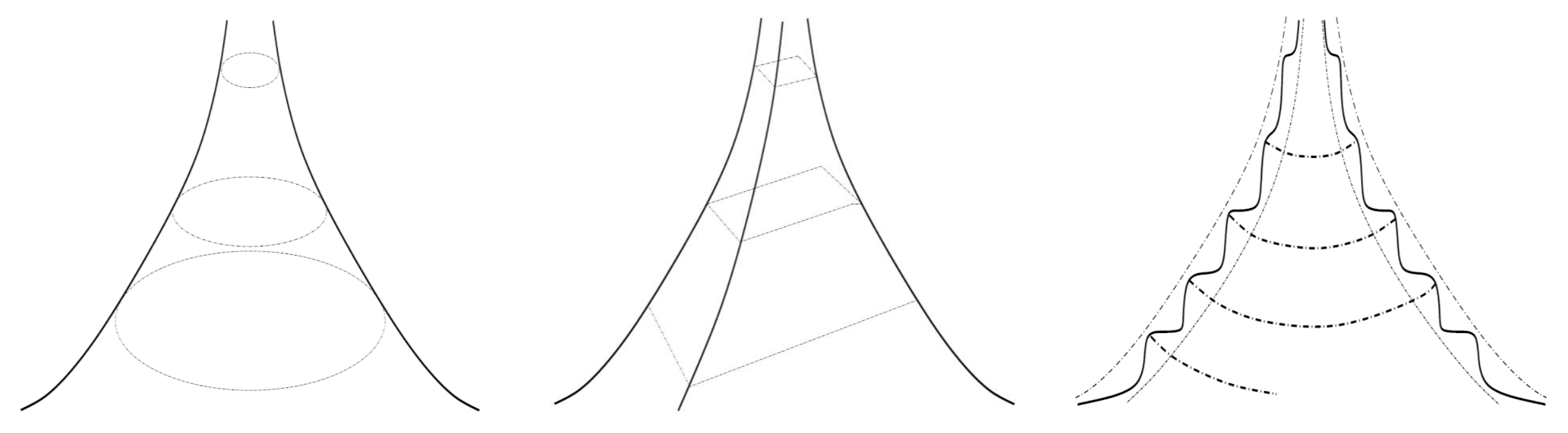}};
			
			
		\end{tikzpicture}
		
	}
	\caption{Kernels for: the $s$-perimeter, the anisotropic $s$-perimeter, more general $\mathcal L_2$ kernels.}
\end{figure}

A set $E$ is said to minimize the $s$-perimeter in $\Omega$ if
 \begin{equation}\label{minimizer}
P_{s,\Omega}(E) \le P_{s,\Omega}(F), \quad \mbox{ for all }F  \mbox{ with }E\setminus \Omega  = F\setminus \Omega.
\end{equation}
The (boundaries of the) minimizers of the $s$-perimeter
are often called nonlocal minimal (or $s$-minimal) surfaces.

In this paper, we  study a
more general functional, in which the interaction
kernel is not necessarily homogeneous and 
rotational invariant.
We consider  a kernel~$K$ satisfying
\begin{equation}\label{Knonnegative}
K(z) \ge 0,
\end{equation}
\begin{equation}\label{Keven}
K(z) = K(-z),
\end{equation}
\begin{equation}\label{Kintegrability}
\int_{\R^n} K(z) \min\{1, |z|\} \,dz< +\infty
\end{equation}
and
\begin{equation}\label{Kboundedbelow}
K \ge 1 \quad \mbox{in } B_{2}.
\end{equation}

To prove our main results we will require the following additional assumption on the first and second derivatives of the kernel $K$: 
\begin{equation}\label{est-second-deriv}
  \max\biggl\{ \,|z|\, |\partial_{e} K(z)| \,,\, |z|^2 \sup_{|y-z|\le |z|/2}\,\, |\partial_{ee} K(y) |\biggr\} \le  K^*(z)
\end{equation}
for all $z\in\R^n\setminus \{0\}$ and for all $e\in S^{n-1}$, for some kernel $K^*$.

Throughout the paper we will have one of the three following cases:
\begin{itemize}
\item  $K^*(z) = C_1K(z)$;
\item  $K^*(z) = C_1 \bigl(  K(z)+ \chi_{\{|z|<R_0\}} (z)\bigr)$ for some $R_0\ge 2$;
\item  $K^*(z)\in L^1(\R^n)$.
\end{itemize}
We emphasize that the kernels of the fractional $s$-perimeter and its anisotropic version satisfy \eqref{est-second-deriv} with $K^*(z)=C_1K(z)$. Therefore, a reader interested in the results for these particular cases, can mentally replace $K^*$ by $C_1\,K$ in all the paper.
We allow the second case of $K^*$ in order to obtain results for compactly supported kernels, as for example
$(9-|z|^2)^3_+|z|^{-n-s}$.
With the third case,  we will be able to obtain strong results for nonsingular kernels like $e^{9-|z|^2}$.

We set
\[ L_K(A,B) = \int_A\int_B K(x-\bar x) \,dx\,d\bar x.\]

We define, for a measurable set $E\subset \R^n$, the $K$-perimeter of $E$ in $\R^n$ as
$$P_K(E)=L_K(E,\mathcal C E).$$

We define the $K$-perimeter of $E$ inside $\Omega$, $P_{K,\Omega}(E)$ similarly as in \eqref{def-nlps} with $L_K$ replacing $L_s$. That is,
\begin{equation}\label{def-nlp} P_{K,\Omega}(E):=L_K(E\cap \Omega, \mathcal C E\cap \Omega) + L_K(E\cap \Omega, \mathcal C E\setminus \Omega) + L_K(E\setminus \Omega,  \mathcal C E\cap \Omega).
\end{equation}
Note that our definition of $P_{K,\Omega}(E)$ agrees with the one of $P_K(E,\Omega)$ given in \cite[Section 3]{FFMMM}.
\begin{rem}\label{integralK}
We observe that if $K$ satisfies \eqref{Kintegrability}, then every Lipschitz bounded domain $U$ has finite $K$-perimeter in $\R^n$. Indeed,
\[
\begin{split}
P_K(U)&=\int_U\int_{\mathcal C U}K(\bar x-x)dx d\bar x=\int_{\R^n}dz\int_{U\cap (\mathcal C U-z)}dx\, K(z)\\
&=\int_{\R^n}|U\setminus (U-z)|K(z)dz
\leq C\int_{\R^n}\min\{1,|z|\}K(z)dz<\infty,
\end{split}
\]
where we have used the change of variables $z=\bar x -x$ and Fubini Theorem.
\end{rem}
We next formally state the definition of minimizer of the $K$-perimeter.
\begin{defi}\label{minim}
We say that $E$ is a \emph{minimizer} for $P_{K,\Omega}$ in an open bounded set $\Omega$, if $P_{K,\Omega}(E)<\infty$ and
$$P_{K,\Omega}(E)\leq P_{K,\Omega}(F)$$
for any set $F$ which coincides with $E$ outside $\Omega$, that is $F\setminus \Omega=E\setminus \Omega$.
\end{defi}

We also define the notion of stable  set for the $K$-perimeter. 
\begin{defi}\label{stable}
We say that $E$ is a \emph{stable  set} for $P_{K,\Omega}$ if $P_{K,\Omega}(E)<\infty$ and for any given vector field  $X =X(x,t)\in C^2_c(\Omega\times(-1,1);\R^n)$ and $\varepsilon>0$ there is $t_0>0$ such that the following holds. Denoting $F_t = \Psi_t(E)$, where $\Psi_t$ is the integral flow of $X$, we have
$$0\leq  P_{K,\Omega}(F_t\cup E)- P_{K,\Omega}(E) + \epsilon t^2$$
and
$$0\leq   P_{K,\Omega}(F_t\cap E)- P_{K,\Omega}(E) + \epsilon t^2$$
for all $t\in(-t_0,t_0)$.
\end{defi}

For our second theorem we will consider kernels $K$ in the class $\mathcal L_2(s,\lambda, \Lambda)$ introduced by Caffarelli and Silvestre in \cite{CS-reg} (see Figure 1). Namely, the kernels $K(z)$ satisfying \eqref{Keven},
\begin{equation}\label{L0}
\frac{\lambda}{|z|^{n+s}} \le K(z) \le \frac{\Lambda}{|z|^{n+s}}
\end{equation}
and
\begin{equation}\label{L2}
  \max\bigl\{|z|\, |\partial_{e} K(z)| \,,\, |z|^2 |\partial_{ee} K(z) |\bigr\} \le \frac{\Lambda}{|z|^{n+s}}
\end{equation}
for all $z\in\R^n\setminus \{0\}$ and for all $e\in S^{n-1}$. 
Note that, after multiplying a kernel $K\in \mathcal L_2$ by a positive constant, we may assume that $\lambda \geq 2^{n+s}$ and hence $K$ satisfies \eqref{Knonnegative}--\eqref{est-second-deriv} with $K^* = C_1 K$.

A very relevant particular case to which our results apply is that of $s$-fractional anisotropic perimeters, introduced in \cite{L}. This case corresponds to the
choice of the kernel
\begin{equation} \label{KofformH}
K(z) = \frac{a(z/|z|)}{|z|^{n}},
\end{equation}
where $a$ is some positive, even  $C^2$ function on the $(n-1)$-dimensional unit sphere $S^{n-1}$ (see Figure 1).
The notion of anisotropic nonlocal perimeter was considered in \cite{L}, where some asymptotic results for $s\rightarrow 1^-$ where established.

\subsection{Motivations of nonlocal perimeters}\label{R:W:M}

To favor a concrete intuition of the nonlocal perimeter
functional, we now recall some practical applications
of the nonlocal perimeter functionals. In these applications,
it is also natural to consider interactions that
are not homogeneous or rotationally invariant.

\textbf{A.}
The first application that we present
is related to image processing and bitmaps.

Let us consider the framework of BMP type images  with square pixels of (small) size~$\rho>0$
(and suppose that~$1/\rho\in\N$ for simplicity).
For simplicity, let us consider a picture of a square of side~$1$,
with sides at~$45^\circ$ with respect to the orientation
of the pixels and let us compare
with the ``version'' of the square
which is represented in the image (see Figure~2).

In this configuration, the classical perimeter functional
provides a rather inaccurate tool to analyze this picture,
no matter how small the pixels are, i.e. no matter
how good is the image resolution.

Indeed, the perimeter of the ideal square is~$4$,
while the perimeter of the picture displayed by the monitor
is always~$4\sqrt{2}$ (independently on the smallness of~$\rho$),
so the classical perimeter is always producing an error
by a factor~$\sqrt{2}$, even in cases of extremely high resolution.

Instead, the fractional perimeter (for instance with $s=0.95$) or other nonlocal perimeters
would provide a much better approximation of the classical perimeter of the ideal square in the case
of high image resolution.
Indeed, the discrepancy~$D_s(\rho)$ between the $s$-perimeter of the ideal square  and
the $s$-perimeter of the pixelled square
is bounded by above by the sum of the interactions
between the ``boundary pixels'' with their complement:
these pixels are the ones which intersect the boundary
of the original square, and their number is~$4/\rho$.

By scaling, the interaction of one pixel with its complement
is of the order of~$\rho^{2-s}$, therefore we obtain that~$D_s(\rho)\le
C\rho^{1-s}$, which is infinitesimal as~$\rho\to0$.

Since the fractional perimeter (suitably normalized) is close to the classical one as $s\to1^-$,
that the fractional perimeter provides in this case a more precise information that the classical one.

\begin{figure}[h]
	\centering	
	\resizebox{10cm}{!}{
	
		\begin{tikzpicture}[scale=1]
		
			\node (myfirstpic) at (0,0) {\includegraphics[scale=0.6]{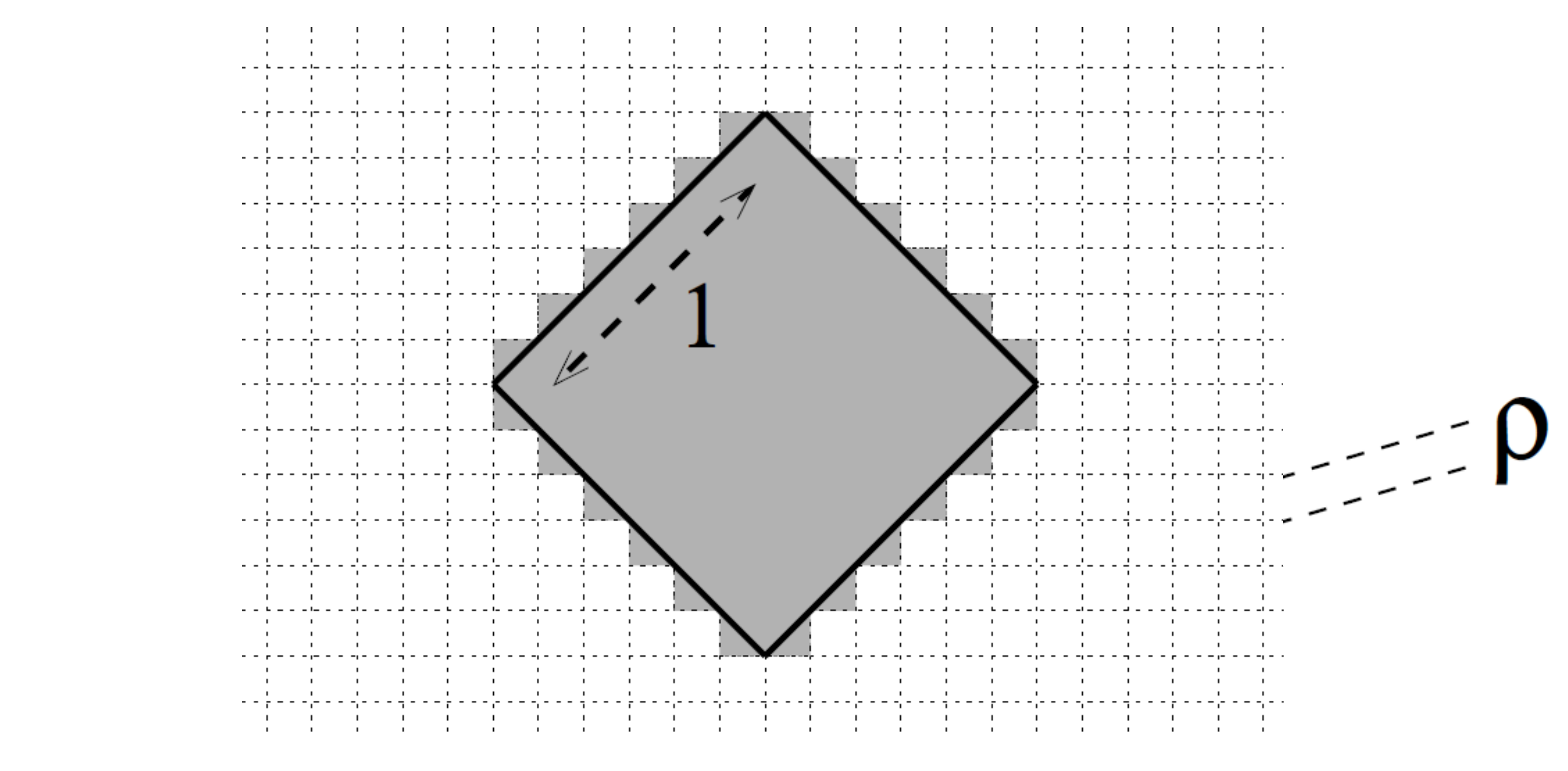}};
			
			
		\end{tikzpicture}
		
	}
	\caption{Discrepancy of local/nonlocal perimeters in a bitmap.}
\end{figure}

{\bf B.}
Another main motivation for the study of nonlocal $s$-minimal surfaces, as explained in \cite{CRS}, is the understanding of steady states for nonlinear interface evolution processes with L\'evy diffusion.
Namely let us think of $u(t,\,\cdot\,): \R^n \rightarrow [0,1]$ as representing the state at time $t$ of some interface phenomenon where two stable states $u\equiv 1$ and $u\equiv 0$ diffuse and ``compete'' to conquer the whole space. In concrete applications $u$ could be, for instance, the density of an invasive biological specie.

For a wide class of such situations, the evolution equation that governs $u$ is of the type
\[ u_t + \mathcal Lu = f(u),\]
where $\mathcal L$ is a ``diffusion operator'' ---e.g. $\mathcal L =(-\Delta)^{s/2},\  s\in (0,2]$--- and $f$ is a bistable nonlinearity with $f(0)=f(1)=0$ and $f(z)$ increasing (resp. decreasing) near $z=0$ (resp. $z=1$).

An extreme version of this evolution process, heuristically corresponding to a huge balanced $f$ like $f(u)= M\bigl((2u-1)-(2u-1)^3\bigr)$ with $M\gg1$ is the following.

Given an open set $E\subset \R^n$ with smooth boundary we define its density function $u$ by
\[  u(x) = \lim_{r\searrow 0} \frac{|B_r(x)\cap E|}{|B_r(x)|}.\]
That is $u(x)$ takes the values $1$, $1/2$ or $0$ depending on whether $x$ belongs to $E$, $\partial E$ or the interior of $\mathcal CE$

Let $\mathcal L$ be a ``diffusion operator'', or more rigorously, an infinitesimal generator of a L\'evy process. For $t \in \tau \mathbb N\cup\{0\}$, where $\tau$ is a tiny time step, we  define the discrete in time evolution $\Phi^{\mathcal L}_{t}(u)$ of the density function $u$ of $E$ as follows:
\[
\Phi^{\mathcal L}_{t+\tau}(u)(x)=
\begin{cases}
1   \quad &  \mbox{if }v(\omega, x)>1/2 \\
1/2          &  \mbox{if }v(\omega, x)= 1/2 \\
0             &  \mbox{if }v(\omega, x)<1/2,
\end{cases}
\]
where $\omega=\omega(\tau)$ is an appropriate time step depending on $\tau$ and $v$ is the solution to
\[ v_t+\mathcal Lv =0\quad \mbox{ with initial condition } v(0,\cdot ) = \Phi^{\mathcal L}_t (u).\]

In this way $\Phi^{\mathcal L}_t$ defines a discrete in time surface evolution of $\partial E$ ---excluding patological cases in which thickening of the set $\{v=1/2\}$ might occur.

Heuristically, a set $E$ with smooth enough boundary will be stationary under the flow $\Phi^{\mathcal L}_{t}$ (with infinitesimal $\tau$) if and only if its density function $u$ satisfies
\begin{equation}\label{euler-lagrange}
 {\mathcal L} u(x) = 0 \quad \mbox{for all }x\in \partial E = \{u=1/2\}.
\end{equation}
Indeed, in this way the evolution $v_t+{\mathcal L}v =0$ will be $v_t\approx 0$ on $\partial E$ for $0<t<\omega(\tau)\ll 1$, and the boundary points will not move.
Note that this heuristic argument is independent of the modulus of continuity $\omega$.

In some cases, under an appropriate choice of $\omega =\omega(\tau)$  the discrete flow $\Phi^{\mathcal L}_t$ can be shown to converge to some continuous flow as $\tau \searrow 0$.

When ${\mathcal L} = -\Delta$ is the Laplacian, under the choice $\omega= \tau$ the $\Phi^{\mathcal L}_t$ converges to the mean curvature flow. This classical result was conjectured by Merriman, Bence, and Osher in \cite{MBO}, and proven to be true by Evans \cite{E} and Barles and Georgelin \cite{BG}. In \cite{CN}, Chambolle and Novaga generalized this result to the case of anisotropic and crystalline curvature motion. In \cite{IPS} Ishii, Pires and Souganidis study the convergence of general threshold dynamics type approximation schemes to hypersurfaces moving with normal velocity depending on the normal direction and the curvature tensor.

Finally, in \cite{CSou}, the case ${\mathcal L} = (-\Delta)^{s/2}$ was considered: in this case $\Phi^{\mathcal L}_t$  still converges to the mean curvature flow for $s\in(1,2)$ with $\omega = \tau^{s/2}$ and for $s=1$ with
$\omega$ implicitly defined $\omega^2|\log \omega| = \tau$ for $\tau$ small. 

Instead,  for $s\in(0,1)$
and $\omega = \tau^{s/(1+s)}$, the discrete flow $\Phi^{\mathcal L}_t$  with ${\mathcal L} = (-\Delta)^{s/2}$ converges to a new geometric flow:  the  $s$-nonlocal mean curvature flow (see \cite{CSou}, Theorem 1) ---a flow where the normal displacement is proportional to the nonlocal mean curvature; see also \cite{CMP, I, SaezV}. Fractional $s$-minimal surfaces are stationary under this $s$-nonlocal mean curvature flow.

At the level of discrete flow, we can replace $(-\Delta)^{s/2}$ with 
 a more general elliptic operator of form
\begin{equation}\label{ell}
 {\mathcal L}u(x) = \int_{\R^n} \bigl( u(x)- u(\bar x)\bigl) K(x-\bar x)\, d\bar x
 \end{equation}
where $K$ satisfies \eqref{Knonnegative}-\eqref{Kboundedbelow}.
Heuristically,  minimizers of the $K$-perimeter should be natural candidates to being  stationary under the flow $\Phi^{\mathcal L}_{t}$ as $\tau\rightarrow 0$.

\textbf{C.} Another motivation for the study of nonlocal $s$-minimal surfaces comes from models describing phase-transitions problems with long-range interactions. In the classical theory of phase transitions, one consider the energy functional 
\begin{equation}\label{energyAC}
\mathcal E(u)=\int_\Omega \varepsilon^2 |\nabla u|^2 + W(u),
\end{equation}
where $W$ is a double well potential representing the dislocation energy, and the first term, involving $\nabla u$, penalizes the formation of unnecessary interfaces. The classical $\Gamma$-convergence result by Modica and Mortola \cite{MM} states that the energy functional $\varepsilon^{-1}\mathcal E$ $\Gamma$-converges to the (classical) perimeter functional. A nonlocal analogue of \eqref{energyAC} is the following
$$\mathcal E_\sigma(u)= \varepsilon^{2\sigma} \mathcal K_\sigma(u,\Omega)+ \int_\Omega W(u),$$
where
$$\mathcal K_\sigma(u,\Omega):=\frac{1}{2}\int_\Omega\int_\Omega\frac{|u(x)-u(\bar x)|^2}{|x-\bar x|^{n+2\sigma}}dxd\bar x + \int_\Omega\int_{\mathcal C\Omega}\frac{|u(x)-u(\bar x)|^2}{|x-\bar x|^{n+2\sigma}}dxd\bar x.$$
The previous energy functional models long range (or nonlocal) interactions between the particles ---the density of particles at a point is influenced by the density at other points that may be not infinitesimally close.  The minimizers of the functional $\mathcal E_\sigma(u)$ have been studied in several recent papers \cite{CC1,CC2,CSM,CS1,CS2,SireV}. A list of results established in these works includes: 1-D symmetry in low dimensions, energy estimates, Hamiltonian identities, existence and decay properties of 1-D solutions, etc.

In \cite{SV-gamma}, Savin and one of the authors study the $\Gamma$-convergence of the  energy functional $\mathcal E_\sigma$.
In particular, they prove that when $\sigma\in [1/2,1)$, after a suitable rescaling, $\mathcal E_\sigma$ $\Gamma$-converges to the classical perimeter functional. On the other hand, when $\sigma\in(0,1/2)$ the functional $\varepsilon^{-2\sigma}\mathcal E_\sigma$ $\Gamma$-converges to the nonlocal $s$-perimeter with $s=2\sigma$. 
Note that  $\varepsilon^{-2\sigma}\mathcal E_\sigma=\mathcal K_\sigma + \varepsilon^{-2\sigma}\int_{\Omega}W(u)$ and thus, in this renormalization, there is no small coefficients in front of the Dirichlet energy.

Analogously, we could consider more energy functionals of the form
\begin{equation*}
\begin{split}\mathcal E_K(u)&= \frac{1}{2}\int_\Omega\int_\Omega|u(x)-u(\bar x)|^2 K(x-\bar x) dxd\bar x \\
&\hspace{1em} + \int_\Omega\int_{\mathcal C\Omega}|u(x)-u(\bar x)|^2 K(x-\bar x) dxd\bar x+ M \int_\Omega W(u),
\end{split}
\end{equation*}
where $M\gg 1$ is a large real number.

Heuristically, similarly to the result of \cite{SV-gamma}, minimizers of $\mathcal E_K$ should ``converge'' to minimizers of the $K$-perimeter $P_K$ as $M\rightarrow \infty$.

\subsection{Statement of the main results}\label{RIS:S}

From now on, we will assume that $K$ satisfies assumptions \eqref{Knonnegative}--\eqref{est-second-deriv}.

We state our main results in the two following subsections. In the first one we give the uniform $BV$-estimates for stable sets, and their consequence on existence and compactness of minimizers of the $K$-perimeter ---see \cite{EvGa, Giusti, Mag}.
In the second one we state our quantitative flatness results and we comment on some corollaries and some applications for specific choices of the kernels that are of independent interest.

\subsubsection{Uniform $BV$-estimates}
We recall the (classical) notion of $BV$-space and of sets of finite perimeter. Let $\Omega$ be an open set of $\R^n$. Given a function $u$ in $L^1(\Omega)$, the \textit{total variation} of $u$ in $\Omega$ is defined as follows:
$$|\nabla u|(\Omega):=\sup \left\{\int_\Omega u \,\mbox{div}\phi\,
{\mbox{ with }}\, \phi \in C_c^1(\Omega,\R^n),\;|\phi|\leq 1\right\}.$$
Here, and throughout the paper, we denote $C^1_c(U;A)$ the  $C^1$ vector fields compactly supported in $U$ and taking values in $A$.

The space $BV(\Omega)$ is defined as the space of functions which belong to $L^1(\Omega)$ and have $|\nabla u|(\Omega)$ finite.
Moreover, we say that a set $E\subset \R^n$ has \textit{finite perimeter} in $\Omega$, when the distributional gradient $\nabla \chi_E$ of its characteristic function is a $\R^n$-valued Radon measure on $\R^n$ and $|\nabla \chi_E|(\Omega)<\infty$. In this case, we define the perimeter of $E$ in $\Omega$ as:
$${\rm{Per}}_\Omega(E)=|\nabla\chi_E|(\Omega).$$
Finally, we define the \textit{reduced boundary} $\partial^*E$ of a set of finite perimeter $E$ as follows: $\partial^*E$ is the set of all points $x$ such that $|\nabla \chi_E|(B_r(x))>0$ for any $r>0$ and
\begin{equation}\label{normal}
\lim_{r\rightarrow 0^+}\frac{\nabla \chi_E(B_r(x))}{|\nabla \chi_E(B_r(x))|}\quad \mbox{exists and belongs to}\;\;S^{n-1}.
\end{equation}
For any $x\in \partial^*E$, we denote by $-\nu_E(x)$ the limit in \eqref{normal} and we call the Borel vector field $\nu_E:\partial^*E\rightarrow S^{n-1}$ the \textit{measure theoretic outer unit normal} to $E$.

The following are our main results:

\begin{thm}[\textbf{$BV$-estimates for stable  sets}]\label{thmstable}
Let $n\ge 2$. Let $E$ be a stable  set of the $K$-perimeter in $B_4$, with $K$ in $\mathcal L_2(s,\lambda, \Lambda)$, that is, with $K$ satisfying \eqref{Keven}, \eqref{L0}, and \eqref{L2}.

Then, the classical perimeter of $E$ in $B_1$ is finite. Namely $\chi_E$ belongs to $BV(B_1)$ with the following universal estimate
\[{\rm Per}_{B_1}(E)= |\nabla\chi_E| (B_1) \le  C(n,s,\lambda, \Lambda).\]
\end{thm}


Rescaling Theorem \ref{thmstable}, and using an interpolation inequality that relates $P_K$ and ${\rm Per}$, we obtain
\begin{cor}\label{thmstable1}
Let $n\ge 2$. Let $E$ be stable  set of the $K$-perimeter in $B_{4R}$, with $K$ in $\mathcal L_2(s,\lambda, \Lambda)$, i.e. satisfying condition  \eqref{Keven}, \eqref{L0} and
\eqref{L2}.
Then,
\begin{equation}\label{est-P}
 {\rm{Per}}_{B_R}(E)\le  C(n,s,\lambda, \Lambda) R^{n-1}.
 \end{equation}

As a consequence
\begin{equation}\label{est-P_K}
 P_{K,B_R}(E)\le  C(n,s,\lambda, \Lambda) R^{n-s}.
 \end{equation}
\end{cor}

We observe that the exponent $n-s$ in \eqref{est-P_K} is optimal since it is achieved when $E$ is an halfspace.
To prove \eqref{est-P_K} when $E$ is minimizer, it is enough to compare the $K$-perimeter of $E$ with the $K$-perimeter of $E \cup B_R$.
However, for \emph{stable} stationary sets this simple comparison argument can not be done and  the proof is much more involved ---we need to prove our (stronger) uniform $BV$-estimates and deduce  \eqref{est-P_K} as a byproduct.

Theorem \ref{thmstable} follows from the following result for general kernels combined with an appropriate scaling and covering argument.
\begin{thm}\label{BV-est}
Let $n\ge 2$. Let $E$ be a stable  set of the $K$-perimeter in $B_4$, with $K$ satisfying \eqref{Knonnegative}-- \eqref{est-second-deriv}.
If $P_{K^*,B_4}(E)<\infty$, then  the classical perimeter of $E$ in $B_1$ is finite. Namely $\chi_E$ belongs to $BV(B_1)$ with the following estimate 

\[ {\rm Per}_{B_1}(E)=|\nabla\chi_E| (B_1) \le  \sqrt 2\, n\,\sqrt{P_{K^*,B_4}(E)}+ |S^{n-1}|.\]
Here,  $|S^{n-1}|$ denotes the $(n-1)$-dimensional measure of the sphere $S^{n-1}$.
\end{thm}

Theorem \ref{BV-est} can be applied to several particular cases. We state below the ones which we consider more relevant.
\begin{cor}\label{BV-L^1}
Let $n\ge 2$. Let $E$ be a stable  set of the $K$-perimeter in $B_4$, with $K$ satisfying \eqref{Knonnegative}-- \eqref{est-second-deriv} and $K^*\in L^1(\R^n)$.
Then,  

\[ {\rm Per}_{B_1}(E)=|\nabla\chi_E| (B_1) \le  \sqrt 2\,n\,|B_4|^{1/2}\,||K^*||_{L^1(\R^n)}^{1/2}+ |S^{n-1}|.\]
\end{cor}
Recall that we denote the $K$-perimeter of a ball $B_R$ (relative to $\R^n$) as
\begin{equation}\label{defPi}
 P_K(B_R) = \int_{B_R}\int_{\mathcal C B_R} K(\bar x-x)\,d\bar x\,dx.
\end{equation}
We remark that for kernels as in~\eqref{KofformH} we have that
$P_K(B_R)=CR^{n-s}$.
Notice also that, by a simple comparison argument,
\[\sup\bigl\{ P_{K,B_R}(E) \,:\, E \mbox{ minimizer of the $K$-perimeter in $B_R$}\bigr\}  \le P_K(B_R).\]
Indeed, if $E$ is a minimizer in $B_R$, then  $P_{K,B_R}(E)\le P_{K,B_R}(E\cup B_R) \le P_K(B_R)$.

When $E$ is a minimizer and $K^*=C_1(K+\chi_{|z|<R_0})$, then $P_{K^*,B_R}(E)$ can be bounded by above by $CP_K(B_R)$. This is the content of the following proposition (which is proven later on in Section 5).
\begin{prop}\label{est-P_K^*}
Let $E$ be a minimizer of the $K$-perimeter in $B_R$ with $R\geq 1$ and $K$ satisfying \eqref{Knonnegative}-- \eqref{est-second-deriv}, and $K^*=C_1(K+\chi_{|z|<R_0})$ for some $R_0\geq 2$.

Then,
\begin{equation*}
P_{K^*,B_R}(E)\leq CC_1 P_K(B_R),
\end{equation*}
where $C$ is a constant depending only on $n$ and $R_0$.
\end{prop}
As a consequence of Theorem \ref{BV-est} and Proposition \ref{est-P_K^*}, we deduce the following
\begin{cor}[\textbf{$BV$-estimates for minimizers}]\label{thm0}
Let $E$ be a minimizer of the $K$-perimeter in $B_4$, with $K$ satisfying \eqref{Knonnegative}-- \eqref{est-second-deriv}, and $K^*=C_1(K+\chi_{|z|<R_0})$ for some $R_0\geq 2$.

Then, 
\[ {\rm Per}_{B_1}(E)=|\nabla\chi_E| (B_1) \le  C\sqrt{C_1 P_K(B_4)},\]
where $C$ is a constant depending only on $n$ and $R_0$.
\end{cor}

\begin{figure}[h]
	\centering	
	\resizebox{10cm}{!}{
	
		\begin{tikzpicture}[scale=1]
		
			\node (myfirstpic) at (0,0) {\includegraphics[scale=0.6]{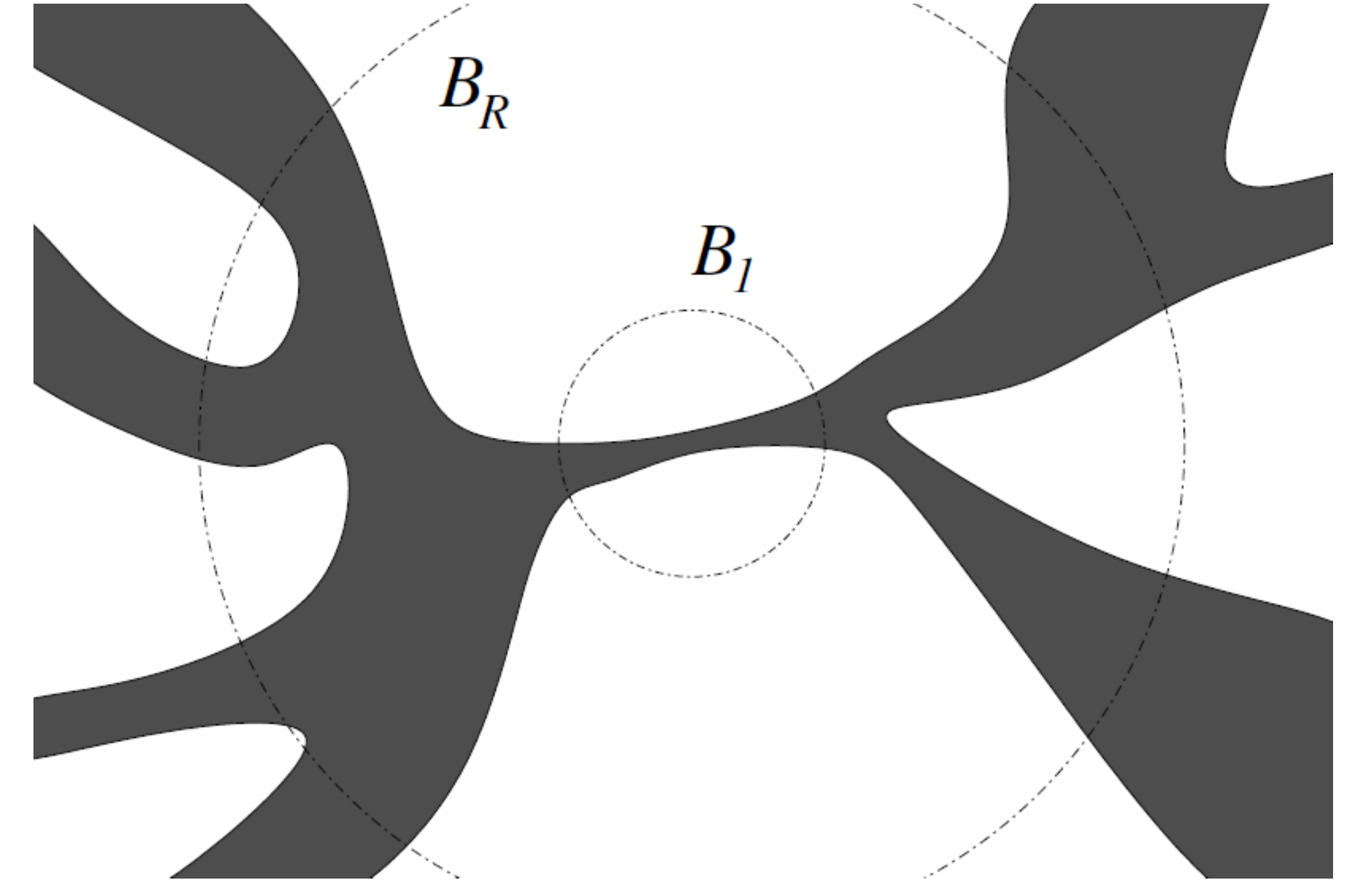}};
			
			
		\end{tikzpicture}
		
	}
	\caption{A minimizer of the $K$-perimeter in $B_R$ has finite classical perimeter in $B_1$.}
\end{figure}

As explained in the beginning of the introduction, the ``a priori''  $BV$-estimate established in Corollary \ref{thm0} allows us to prove a 
very general 
existence result 
for minimizers of $P_{K,\Omega}$. We state it next.
\begin{thm}[\textbf{Existence of minimizers}]\label{existence}
Let $\Omega$ be a bounded Lipschitz domain, and $E_0\subset \cC\Omega$ a given measurable set. Suppose that $K$ satisfies assumptions \eqref{Knonnegative}-- \eqref{est-second-deriv}. Then, there exists a set $E$, with $E\cap \cC\Omega=E_0$ that is a minimizer for $P_{K,\Omega}$ ---in particular $P_{K,\Omega}(E)<\infty$.
\end{thm}

Notice that when $K$ belongs to $L^1(\R^n)$, then
$P_{K,\Omega}(F)<|\Omega|\int_{\R_n}K<\infty$ for {\em all} measurable sets $F$. Thus, in principle, we do not have compactness for sequences of sets with uniformly bounded $K$-perimeter.

The idea of the proof of Theorem \ref{existence}
(which will be given in Section 4) consists in considering the ``singularized'' kernel
$$K_\epsilon(z):=K(z)+\frac{\epsilon}{|z|^{n+\frac{1}{2}}},$$
which, for every fixed $\epsilon$, admits a minimizer $E_\epsilon$ by the standard compactness of $H^{\frac{1}{4}}$ in $L^1$. All the new kernels $K_\epsilon$ satisfy assumptions \eqref{Knonnegative}--\eqref{est-second-deriv} with constants that are uniform in $\epsilon$. Thus,  Theorem \ref{thm0} gives uniform $BV$-bounds for the characteristic functions of the minimizers $E_\epsilon$. These bounds give the necessary compactness in $L^1$ to prove the existence of a limiting set as $\epsilon \rightarrow 0$. In order to prove that the limiting set is a minimizer of $P_{K,\Omega}$, we use some other important ingredients (such as a nonlocal coarea formula and a density result for smooth sets into sets of finite $K$-perimeter) that will be established later on in Section 6.

\subsubsection{Quantitative flatness results}
Our quantitative flatness results in low dimensions $n=2,3$ state that (under appropriate assumptions on the kernel $K$) a stable set $E$ of the $K$-perimeter in a very large ball $B_R$ is close to being a flat graph in $B_1$. Namely, for some $\varepsilon=\varepsilon(R)$ that decreases to $0$ when $R$ increases to $\infty$, and after a rotation of coordinates, the following three properties hold.

\label{FFPA}

\begin{enumerate}
\item[({\bf F1})]  For some $t\in [-1,1]$,
 \[|(E \triangle \{x_n\le t\})\cap B_1 |  \le \varepsilon,\]
where $\triangle$ denotes the symmetric difference.
\vspace{4mm}

\item[({\bf F2})]  There is a set $\mathcal B \subset B_{1}^{(n-1)}=\{x' \in \R^{n-1}\,:\,|x'|\le 1\}$ with $|\mathcal B| \le \varepsilon$ such that 
\[
(E\cap B_1) \setminus (\mathcal B\times \R)  = \bigl\{(y,x_n)\in B_1 \,: \, x_n \le g(y),\ y \in (B_{1}^{(n-1)}\setminus\mathcal B) \bigr\}.
\]
for some measurable function $g: B_{1}^{(n-1)}\rightarrow[-1,1]$.
\vspace{4mm}

\item[({\bf F3})]  Denoting  $F^{\varepsilon} = \{(x',x_n/\varepsilon)\,:\, (x',x_n)\in F\}$, we have
\[ {\rm Per}_{B_1^\varepsilon}(E^\varepsilon) \le C(n),\]
where $C(n)$ is a constant depending only on the dimension $n\in \{2,3\}$.
\end{enumerate}
\vspace{2mm}

Point ({\bf F1}) says that the set $E$ is close in the $L^1$-sense to being a half-plane while
point ({\bf F2}) says that  $\partial E\cap B_1$ is  a graph after removing ``vertical'' cylinders of small measure (see Figure 4).
Moreover, ({\bf F3})  gives a uniform bound for the classical perimeter of rescalings of $E$ in the vertical direction by a large factor $1/\varepsilon$.
\begin{figure}[h]\label{figure-flat}
	\centering	
	\resizebox{15 cm}{!}{
	
		\begin{tikzpicture}[scale=1]
		
			\node (myfirstpic) at (0,0) {\includegraphics[scale=0.6]{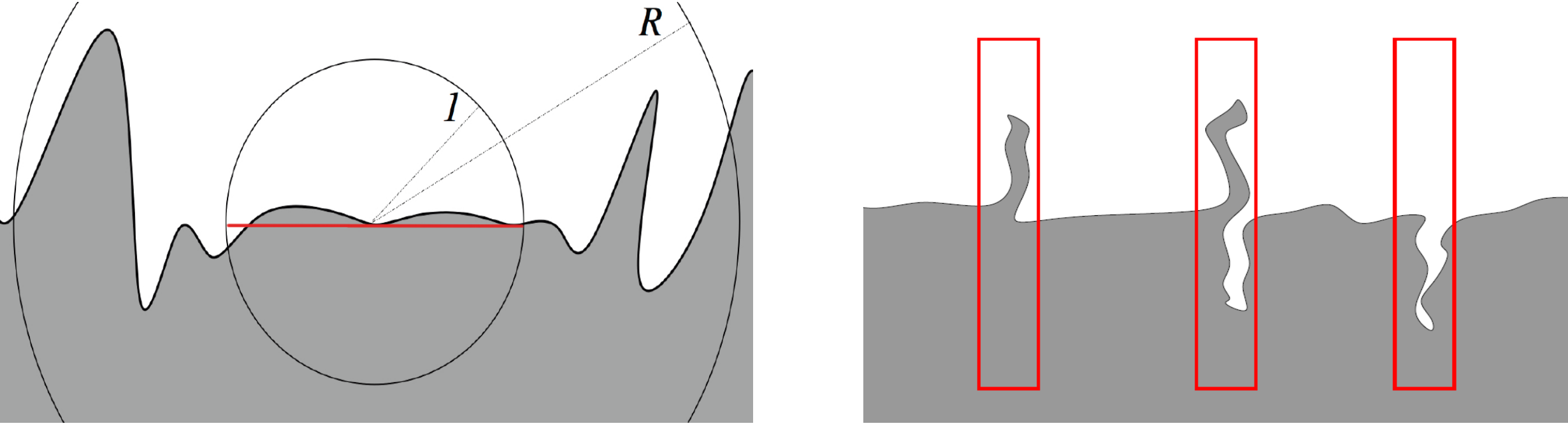}};
			
			
		\end{tikzpicture}
		
	}
	\caption{A minimizer or stable set of the $K$-perimeter in $B_R$ is ``almost'' a flat graph in $B_1$.}
\end{figure}

We give below our quantitative flatness result for stable sets in dimension $n=2$. 
\begin{thm}[\textbf{Flatness for stable sets in dimension two}]\label{aniso}
Let $n=2$. Let $K$ be a kernel belonging to the class $\mathcal L_2(s,\lambda,\Lambda)$, i.e. satisfying conditions \eqref{Keven}, \eqref{L0}, and
\eqref{L2}. Let $E$ is a stable set of the $K$-perimeter in $B_R$ with $R\geq 4$.

Then, after a rotation, $E$ satisfies {\rm ({\bf F1})}, {\rm ({\bf F2})}, and {\rm ({\bf F3})} with
$$\varepsilon=C R^{-s/2},$$
where $C$ is a constant depending only on $s,\,n,\,\lambda,\,\Lambda$.
\end{thm}
\begin{rem}
We recall that Theorem \ref{aniso} applies, in particular, to the fractional anisotropic perimeter introduced in \cite{L}, where 
$$K(z)=\frac{a(z/|z|)}{|z|^{n+s}},$$
with $a\in C^2(S^{n-1})$ positive.
\end{rem}
For the sake of clarity, let us rephrase the first conclusion of Theorem \ref{aniso} in the following way: \textit{Let $K\in \mathcal L_2$ and let $E$ be a stable set of $P_{K,B_R}$. Then, there exists a halfplane $\mathfrak h$ such that}
$$|(E\triangle \mathfrak h)\cap B_1|\leq CR^{-s/2}.$$

Sending $R\rightarrow\infty$ in Theorem \ref{aniso}, we deduce the following 
\begin{cor}\label{flat}
For $K\in\mathcal L_2$,  half-planes are the only stable  sets in every compact set of $\R^2$.\\
\end{cor}

The local analogue of Corollary \ref{flat} was established in \cite{dCP,FCS}, where the following statement is proved: Any complete stable surfaces in $\R^3$ is a plane. As said above in the Introduction, this classification result for classical stable surfaces is still open in dimensions $n\geq 4$.

As explained previously in the Introduction, our quantitative flatness result for stable sets in Theorem \ref{aniso} generalizes the classification theorem of \cite{SV}, that we recall next.
\begin{thm}(Theorem 1 in \cite{SV})\label{flat-SV}
Let $E$ be a cone that is a minimizer of $P_s$ in every compact set of $\R^2$. Then $E$ is a half-plane.
\end{thm}
Using a blow-down argument and  a monotonicity formula ---see Remark \ref{monotCRS}---,  Theorem \ref{flat-SV} implies that halfplanes are the only minimizers of the $s$-perimeter in every compact set of $\R^2$. 

Moreover, similarly as in the theory of classical minimal surfaces, this classification result has important consequences in the regularity theory for nonlocal $s$-minimal surfaces. In particular, combining Theorem \ref{flat-SV} and the results contained in \cite{BFV, CRS}, one can deduce that any minimizer of the $s$-perimeter is smooth outside of a singular set with Hausdorff dimension at most $n-3$.

Our Theorem \ref{aniso} generalizes Theorem \ref{flat-SV} in three directions. First, our result applies to the more general class of stable sets (we recall that any minimizer is a stable set). Second, we can consider more general kernels in $\mathcal L_2$. Third, our result is a quantitative version of Theorem \ref{flat-SV} in the following sense: instead of assuming that $E$ is a minimizer in every compact set of $\R^2$, we assume that  $E$ is a stable set for $P_{K,B_R}$ with \textit{some} large $R$ and we obtain a quantitative control on the flatness of $E$ in $B_1$, depending on $R$.

We point out that, using the $C^2$ estimates for minimizers  of the $s$-perimeter, and scaling invariance,  the distance of $\partial E$  and some  plane in $B_1$  is bounded by $CR^{-1}$, when $E$ is a minimizer of the $s$-perimeter in $B_R$.
However, since the $C^2$ estimates are proved by compactness, we have no explicit estimates for this constant $C$. Moreover, such an approach clearly fails in case the problem
is not scaling invariant or does not have a regularity theory.
Note that with the techniques of this paper we can obtain results for general kernels that are not scaling invariant and for which the existence of some regularity theory is unclear ---see for instance Corollary \ref{thm1L1}.

\begin{rem}\label{monotCRS}
We emphasize that for the specific case of $K(z)=|z|^{-n-s}$, Caffarelli, Roquejoffre and Savin proved a monotonicity formula for the local energy functional associated to the $s$-perimeter via the so called Caffarelli-Silvestre extension. This monotonicity formula allows them to use a blow-up argument to prove regularity results once one knows that the only nonlocal minimal cones are halfplanes.

On the other hand, as explained above, using the monotonicity formula and Theorem \ref{flat-SV} one proves that halfplanes are the only minimizers of the $s$-perimeter in every compact set of $\R^2$ ---thus extending the classification result from cones to all minimizers. 

In our setting, monotonicity formulas are not available but still we can obtain the same type of classification result as a consequence of our quantitative flatness estimates.
\end{rem}
We will deduce Theorem \ref{aniso} from the following more general result.

\begin{thm}\label{flatness}
Let $n=2,3$. Assume that $E$ is a stable  set for the $K$-perimeter in $B_R$ with $R\geq 4$ and $K$ satisfying \eqref{Knonnegative}--\eqref{est-second-deriv}.

Then, after some rotation,  $E$ satisfies   {\rm ({\bf F1})}, {\rm ({\bf F2})}, and {\rm ({\bf F3})} with
\[  \varepsilon=\varepsilon(R) =C  \min\left\{ \sqrt{\frac{P_{K^*,B_R}(E)}{R^2}}\ ,\ \frac{1}{\sqrt{\log R}}\sup_{\rho \in [1,R]}\sqrt{\frac{P_{K^*,B_\rho}(E)}{\rho^2}}\right\},\]
where $C$ is a constant depending only on $n$.
\end{thm}
The result contained in Theorem \ref{flatness} is very general and it can be applied to several choices of kernels. Below, we list some particular cases that are of independent interest.

For kernels with $K^*\in L^1(\R^n)$, we have the following
\begin{cor}\label{thm1L1}
Let $n=2$. Let $E$ be a stable  set of the $K$-perimeter in $B_R$ with $R\ge 4$ and $K$ satisfying \eqref{Knonnegative}--\eqref{est-second-deriv} with
$K^*\in L^1(\R^n)$.

Then, after some rotation,  $E$ satisfies   {\rm ({\bf F1})}, {\rm ({\bf F2})}, and {\rm ({\bf F3})} with
\[  \varepsilon=\varepsilon(R) = \frac{C}{\sqrt{\log R}}\,|B_1|^{1/2}|\,|K^*||_{L^1(\R^n)}^{1/2},\]
where $C$ is a constant depending only on $n$.
\end{cor}
Moreover, if $E$ is a minimizer for the $K$-perimeter, combining Theorem \ref{flatness} and Proposition \ref{est-P_K^*}, we deduce
\begin{cor}[\textbf{Flatness for minimizers in low dimensions}]\label{thm1dim2}\label{thm1dim3}
Let $n=2,3$. Let $E$ be a minimizer of the $K$-perimeter in $B_R$ with $R\ge 4$ and $K$ satisfying \eqref{Knonnegative}--\eqref{est-second-deriv} with
$K^*=C_1(K+\chi_{|z|<R_0})$ for some $R_0>2$.

Then, after some rotation,  $E$ satisfies   {\rm ({\bf F1})}, {\rm ({\bf F2})}, and {\rm ({\bf F3})} with
\[  \varepsilon=\varepsilon(R) = C\min\left\{ \sqrt{C_1\frac{P_K(B_R)}{R^2}}\ ,\ \frac{1}{\log R}\sup_{\rho \in [1,R]}\sqrt{\frac{C_1P_K(B_\rho)}{\rho^2}}\right\},\]
where $C$ is a constant depending only on $n$ and $R_0$.
\end{cor}

%
%


Finally, as a particular case of Corollary \ref{thm1dim2}, we consider the case of kernels with compact support.
\begin{cor}[\textbf{Quantitative flatness for truncated kernels}]\label{compt-supp}
Let $K$ satisfy \eqref{Knonnegative}--\eqref{est-second-deriv}and suppose that $K$ has compact support.
Let $E$ be a minimizer of the $K$-perimeter in $B_R$ with $R\geq 4$. 

Then, after some rotation, $E$ satisfies  {\rm ({\bf F1})}, {\rm ({\bf F2})}, and {\rm ({\bf F3})} with
\begin{equation}
\varepsilon=
\begin{cases}
C R^{-\frac{1}{2}} \quad & \mbox{if}\;\;n=2,\\
\frac{C}{ \sqrt{\log R}}  & \mbox{if}\;\;n=3,
\end{cases}
\end{equation}
where $C$ is a constant depending only on $n$ and $K$.
\end{cor}
This result comes easily applying Corollary \ref{thm1dim2} and by the following energy estimate which holds for the case of a compactly supported kernel $K$:
\begin{equation*}\label{en-est-compt}P_K(B_R)\leq C R^{n-1}.\end{equation*}
%

As a consequence of Corollaries \ref{thm1L1} and \ref{compt-supp} , we obtain the following
\begin{cor}\label{flat2}
For kernels $K$ satisfying \eqref{Knonnegative}--\eqref{est-second-deriv}, halfspaces are the only minimizers in every compact set of $\R^n$ in the following cases:

\begin{itemize}
\item $n=2$ and $K^*\in L^1(\R^n)$;
\item $n=2,3$ and $K$ with compact support.
\end{itemize}
\end{cor}
%


%

The paper is organized as follows:
 \begin{itemize}
 \item In section 2, we establish some preliminary results that we will use in the proof of our main theorems;
 \item In section 3, we prove Theorems \ref{thmstable} and \ref{BV-est} establishing the uniform $BV$-bounds for  stable  sets and for minimizers;
 \item In section 4, we prove our quantitative rigidity result (Theorems \ref{aniso} and \ref{flatness});
 \item Section 5 is dedicated to some technical lemmas that we need in the proofs of the main results;
 \item In section 6, we give the proof of the existence result (Theorem \ref{existence}).
 \end{itemize}

\section{Preliminary results}
Following an idea in \cite{SV}, we want to consider perturbations of the minimizer $E$ which are translations of $E$ in some direction $\boldsymbol {v}$ in $B_{R/2}$ and coincide with $E$ outside $B_R$. To build these perturbations, we consider the two following radial compactly supported functions:

\begin{equation}\label{phi}
\varphi_R(x)=\varphi(|x|/R)=
\begin{cases}
1 & |x|/R<1/2\\
2-2|x|/R & 1/2 \leq |x|/R<1\\
0 & |x|/R \geq 1.
\end{cases}
\end{equation}
and
%
\begin{equation}\label{phitilde}
\tilde \varphi_R(x)=\tilde \varphi_R(|x|)=
\begin{cases}
1 & |x|<\sqrt R\\
2-2\frac{\log(|x|)}{\log R} & \sqrt R \leq |x|<R\\
0 & |x| \geq R.
\end{cases}
\end{equation}
For $\boldsymbol v\in S^{n-1}$ and $t\in[-1,1]$ we define
\begin{equation}\label{def-psi}
\Psi_{R,t}(x):= x + t\varphi_R (x)\boldsymbol v.
\end{equation}

We set
\begin{equation}\label{def-E+}
E_{R,t} = \Psi_{R,t}(E).
\end{equation}
Throughout the paper, we denote

\begin{equation}\label{def-u+}
u = \chi_E\quad \mbox{and}\quad  u_{R,t}(x) =\chi_{E_{R,t}}= u\bigl(\Psi_{R,t}^{-1}(x)\bigr).
 \end{equation}
Note that these definitions depend upon a fixed unit vector $\boldsymbol v$.

Likewise we define $\tilde \Psi_{R,t}$, $\tilde E_{R, t}$,  $\tilde u_{R,t}$,  with $\tilde\varphi_R$ replacing $\varphi_R$.

We prove now the following lemma, which is the appropriate analogue for the nonlocal functional $P_{K,B_R}$ of Lemma 1 in \cite{SV}.
\begin{lem}\label{lem2A}
Let $n\ge 2$, $R\ge4$, and $K$ a kernel satisfying \eqref{Knonnegative}--\eqref{est-second-deriv}. For every measurable
$E\subset \R^n$ with $P_{K,B_R}(E)<\infty$ we have:

{\rm(a)} For all $t\in(-1,1)$
\begin{equation}\label{eqlemAstables}
P_{K,B_R}(E_{R,t}) + P_{K,B_R}(E_{R,-t}) - 2P_{K,B_R}(E) \le 32  \frac{t^2}{R^2}P_{K^*,B_R}(E),
\end{equation}
where $K^*$ is the kernel appearing in  \eqref{est-second-deriv}.

{\rm (b)} For all $t\in(-1,1)$
\begin{equation}\label{eqlemB}
P_{K,B_R}(\tilde E_{R,t}) + P_{K,B_R}(\tilde E_{R,-t}) - 2P_{K,B_R}(E) \le  \frac{(32\pi t)^2}{\log R} \,\sup_{\rho\in[1,R]}  \frac{P_{K^*,B_\rho}(E)}{\rho^2} ,
\end{equation}
where $K^*$ is as above.
\end{lem}

\begin{proof}
We set $A_R:=\R^{2n}\setminus (\mathcal C B_R\times \mathcal C B_R)$.

Let us prove first point (a). 
We have
\[
P_{K,B_R}(E_{R,{\pm t}}) =\frac{1}{2}\iint_{A_R} |u(\Psi_{R,\pm t}^{-1}(x))-u(\Psi_{R,\pm t}^{-1}(\bar x))|^2K(x-\bar x)\,dx\,d\bar x\,.
\]

Changing variables $y= \Psi^{-1}_{R,\pm t}(x)$, $\bar y= \Psi^{-1}_{R,\pm t}(\bar x)$ in the integral we obtain
\begin{equation}\label{changeofvars}
P_{K,B_R}(E_{R,{\pm t}}) =  \frac{1}{2}\iint_{A_R} |u(y)-u(\bar y)|^2 K\bigl(\Psi_{R,\pm t}(y)-\Psi_{R,\pm t}(\bar y)\bigr) \,J_{\pm t}(y) \,dy \,J_{\pm t}(\bar y)\,d\bar y\,,
\end{equation}
where $J_{\pm t}$ are the Jacobians which, as proven in  Lemma 1 in \cite{SV}, are
\[ J_{\pm t}(y)={\rm det}(D\Psi_{R,\pm t}(y))= 1 \pm t\partial_{\boldsymbol v} \varphi_R (y).\]

We call
\[ \varepsilon = \varepsilon(y,\bar y, R) := \frac{\varphi(y/R)-\varphi(\bar y/R)}{|y-\bar y|}.\]
Note that, since $\|\varphi\|_{C^{0,1}(\R^n)} =2$, we have
\begin{equation}\label{gradest}
 |\varepsilon| \le 2/R \quad \mbox{and}\quad   |\partial_{\boldsymbol v} \varphi_R | \le 2/R.
\end{equation}
Let
\[z= y-\bar y.\]
By taking $R\ge4$ we may assume $|\varepsilon|\in (0,1/2]$.
Then, by the assumption \eqref{est-second-deriv} on the second derivatives of the kernel we have
\[ K\bigl(z\pm t\varepsilon|z| \boldsymbol v \bigr) = K(z) \pm t\partial_{\boldsymbol v} K(z) \varepsilon |z| + {\boldsymbol e}_{\pm}(y,\bar y, r),\]
where
\begin{equation}\label{e_pm}
 \bigl|{\boldsymbol e}_{\pm}\bigr| \le \frac 12 t^2  K^*(z)\varepsilon^2\end{equation}


Therefore,
\begin{equation}\label{longproduct}
\begin{split}
&\hspace{-10pt}K\bigl(\Psi_{R,t}(y)-\Psi_{R,t}(\bar y)\bigr) J_{t}(y)J_{t}(\bar y)
+ K\bigl(\Psi_{R,-t}(y)-\Psi_{R,-t}(\bar y)\bigr) J_{-t}(y)J_{-t}(\bar y) =
\\
&=\bigl( K(z) + t\partial_{\boldsymbol v} K(z) \varepsilon |z| + {\boldsymbol e}_+\bigr)\bigl( 1 + t\partial_{\boldsymbol v} \varphi_R (y)\bigr)\bigl( 1 + t\partial_{\boldsymbol v} \varphi_R (\bar y)\bigr)
\\
&\hspace{30pt}+
\bigl( K(z) - t\partial_{\boldsymbol v} K(z) \varepsilon |z| + {\boldsymbol e}_-\bigr)\bigl( 1 - t\partial_{\boldsymbol v} \varphi_R (y)\bigr)\bigl(1 - t\partial_{\boldsymbol v} \varphi_R (\bar y)\bigr)
\\
&=2 K(z) + {\boldsymbol e}(y,\bar y, r)
\end{split}
\end{equation}
where
\begin{equation}\label{errorest}
\begin{split}
\bigl| {\boldsymbol e} \bigr|  &
= \bigl| 2t^2 \partial_{\boldsymbol v} K(z) \varepsilon |z| \bigl( \partial_{\boldsymbol v} \varphi_R (y) +  \partial_{\boldsymbol v} \varphi_R (\bar y)\bigr)  + {\boldsymbol e}_{+} + {\boldsymbol e}_{-} \,+
\\
& \hspace{5mm}+ t({\boldsymbol e}_{+}- {\boldsymbol e}_{-}) \bigl( \partial_{\boldsymbol v} \varphi_R (y)+  \partial_{\boldsymbol v} \varphi_R (\bar y)\bigr)\bigr|
+ t^2\partial_{\boldsymbol v} \varphi_R (y)\partial_{\boldsymbol v} \varphi_R (\bar y)\big[2K(z)+{\boldsymbol e}_{+}+{\boldsymbol e}_{-}\big]
\\
&\le  t^2 \left( 2K^* (z) |\varepsilon| \frac 4 R + K^* (z)\varepsilon^2  + K^* (z)\varepsilon^2 \frac 4 R + \frac{4}{R^2}K^*(z)\big(2+t^2 \varepsilon^2\big)\right)
\\
&
\le t^2 K^*(z)\left( \frac{16}{R^2} + \frac{4}{R^2} + \frac{16}{R^3}+\frac{3}{R^2}\right)
\\
& \le \frac{32 t^2}{R^2} K^*(z).
\end{split}
\end{equation}

Here we have used again the assumption \eqref{est-second-deriv} to estimate terms involving first order derivatives of $K$. In addition, we have used the estimate \eqref{e_pm} for $e_{\pm}$, and that $R^{-3}\le R^{-2}/4$ since $R\ge 4$.

Thus, using \eqref{changeofvars}, \eqref{longproduct} and \eqref{errorest}, we have
\[
\begin{split}
 &P_{K,B_R}(E_{R,t})+ P_{K,B_R}(E_{R{-t}}) -2P_{K,B_R}(E) \le\\
 &\hspace{20mm} \le  \frac{16 t^2}{R^2}\iint_{A_R} |u(y)-u(\bar y)|^2 K^*(y-\bar y)\,dy\,d\bar y=\frac{32 t^2}{R^2}P_{K^*,B_R}(E).
\end{split}
\]
This finishes the proof of (a).

The proof of (b) ---i.e. of \eqref{eqlemB}--- is almost identical with the the difference that we use the function $\tilde\varphi_R$ instead of $\varphi_R$. More precisely, we consider
 $\tilde \Psi_{R,\pm}t$,  $\tilde u_{R,\pm t}$, $\tilde E_{R{\pm t}}$ instead of $\Psi_{R,\pm t}$,  $u_{R,\pm t}$, $E_{R,{\pm t}}$.
The only important difference is that now \eqref{gradest} does not hold since
\[ | \nabla \tilde \varphi_R(x)| = \frac{2  \chi_{\{\sqrt R\le |x| \le R\}}}{\log R \,|x|}.\]
Instead we use
\[ \varepsilon(y, \bar y ,R) \le \pi \frac{2}{\log R\, \max\{\sqrt R, \rho\}} \quad \mbox{whenever }  (y,\bar y)\in \R^{2n}\setminus A_\rho,  1\le \rho\le R.\]
Note that $\R^{2n}\setminus A_\rho = \{(y,\bar y)\, :\, y \ge \rho \mbox{ and }\bar y \ge \rho\}$. The factor $\pi$ appears because we need to apply the mean value theorem connecting $y$ and $\bar y$ by a circular arc contained in $\R^{n}\setminus B_\rho$.

Similarly,
\[ \max\bigl\{|\partial_{\boldsymbol v} \varphi_R (y)|\,,\,  |\partial_{\boldsymbol v} \varphi_R (\bar y)|\bigr\} \le \frac{2}{\log R\,\max\{\sqrt R, \rho\}}  \quad \mbox{ for } (y,\bar y)\in \R^{2n}\setminus A_\rho, \ \rho\ge1.\]

Hence, in place of  \eqref{errorest} we obtain
\begin{equation}\label{errorest2}
\bigl|{\boldsymbol e}(y,\bar y, R)\bigr|   \le \frac{32 \pi^2 t^2}{(\log R)^2 \max\{R^2,\rho^2\}} K^*(z) \quad \mbox{ for } (y,\bar y)\in \R^{2n}\setminus A_\rho, \ \rho\ge1.
\end{equation}

Now,  we decompose the domain $A_R$ in  \eqref{changeofvars} as
\[ A_R = A_{\sqrt R} \cup \bigcup_{i=k+1}^{2k} \tilde A_i,\]
where
\[ k\in \mathbb N, \quad  \log_2 R \le 2k < \log_2 R  +2, \quad   \theta^{2k} = R, \quad \mbox{and}\quad  {\tilde{A}_i} = A_{\theta^{i}}\setminus A_{\theta^{i-1}}.\]
Note that $\theta\in(1,2]$.
Using  \eqref{errorest2} and \eqref{changeofvars} with the previous domain decomposition we obtain
\[
\begin{split}
& \hspace{-15mm}P_{K,B_R}(\tilde E_{R,t})+ P_{K,B_R}(\tilde E_{R,{-t}}) -2P_{K,B_R}(E)\ \le \\
 & \le  \frac{32 \pi^2 t^2}{(\log R)^2} \left( \frac{1}{ R}  \iint_{A_{\sqrt R}} |u(y)-u(\bar y)|^2 K^*(y-\bar y)\,dy\,d\bar y\right.\\
 &\hspace{1em} \left. + \sum_{i=k+1}^{2k} \frac{1}{\theta^{2(i-1)}}  \iint_{\tilde A_{i}} |u(y)-u(\bar y)|^2 K^*(y-\bar y)\,dy\,d\bar y\right)
 \\
 &\le  \frac{32 \pi^2 t^2}{(\log R)^2} \left( \frac{1}{ R} P_{K^*,B_{\sqrt R}}(E)  + \sum_{i=k+1}^{2k} \frac{1}{\theta^{2(i-1)}}P_{K^*,B_{\theta^i}}(E)  \right).
\end{split}
\]

Thus denoting $S :=\sup_{\rho\in[1,R]}  \frac{P_{K^*,B_\rho}(E)}{\rho^2} $
\[
\begin{split}
 P_{K,B_R}(\tilde E_R^t)+ P_{K,B_R}( \tilde E_R^{-t}) -2P_{K,B_R}(E) &\le  \frac{32 \pi^2 t^2}{(\log R)^2} \left( S + \sum_{i=k+1}^{2k} \frac{\theta^{2i}}{\theta^{2(i-1)}} S\right)
\\
&\le \frac{32 \pi^2 t^2}{(\log R)^2}  \theta^2 (k+1) S
\\
& \le  \frac{32 \pi^2 t^2}{(\log R)^2}  4 \bigl( \log_2 (4R) +1 \bigr) S
\\
&\le   \frac{ (32\pi t)^2}{\log R} S.
\end{split}
\]
This finishes  the proof of \eqref{eqlemB} ---and thus of (b).
\end{proof}

The following lemma is a key step in the proof of our main results: given a minimizer $E$ and any possible competitor $F$, it allows to ``measure'' the interaction between points in $E\setminus F$  and points in $F\setminus E$ in terms of the difference between the $K$-perimeter of $F$ and the $K$-perimeter of $E$. Here we see that the nonlocality of the functional plays a crucial role.

\begin{lem}\label{lemEFdelta}
Let $E,\:F \subset \R^n$. Assume that $E$ is a minimizer for $P_{K,B_R}$ and that $F$ coincides with $E$ outside of $B_R$, that is, $E\setminus B_R= F\setminus B_R$. Assume moreover that
\begin{equation}\label{near}
 P_{K,B_R}(F) \le P_{K,B_R}(E) +\delta,
\end{equation}
for some $\delta\ge0$.

Then,
\[ 2 L_K(F\setminus E, E\setminus F) \le  \delta.\]
\begin{proof}
Let $C=E\cup F$ and $D= E\cap F$. Note that both $C$ and $D$ coincides with $E$ and $F$ outside of $B_R$. By a direct computation we find that
\begin{equation}\label{claim1}
P_{K, B_R}(C) +  P_{K, B_R}(D) + 2 L_K(F\setminus E, E\setminus F) = P_{K, B_R}(E) +  P_{K, B_R}(F).
\end{equation}
Using \eqref{near} and the minimality of $E$, we deduce
\[
 P_{K, B_R}(E) +  P_{K, B_R}(F) \le 2P_{K, B_R}(E) +  \delta \le  P_{K, B_R}(C)  + P_{K, B_R}(D) +\delta,
\]
which, together with \eqref{claim1}, concludes the proof of the Lemma.
\end{proof}
\end{lem}

It is worth to observe that, in spite of its simplicity,
the identity in~\eqref{claim1} has consequences
that seem to be interesting in themselves,
such as the fact that
minimizers are included one in the other, as stated in the following result:

\begin{lem}[\textbf{Mutual inclusion of minimizers}]\label{MUTU}
Assume that~$E$ and~$F$ are minimizers for~$P_{K,\Omega}$,
with~$E\setminus\Omega=F\setminus\Omega$. Suppose that $K(y)>0$  for $|y|<{\rm diam}(\Omega)$.
Then, either~$E\subseteq F$ or~$F\subseteq E$.
\end{lem}

\begin{proof} The minimality of the sets give that
$$ P_{K,\Omega}(E)\le P_{K,\Omega}(E\cup F)
\;{\mbox{ and }}\;
P_{K,\Omega}(F)\le P_{K,\Omega}(E\cap F).$$
Then, using~\eqref{claim1}, $$ 2L_K(F\setminus E, E\setminus F)
=P_{K,\Omega}(E)+P_{K,\Omega}(F)-P_{K,\Omega}(E\cup F)-P_{K,\Omega}(E\cap F)
\le 0,$$
which implies that one between ~$F\setminus E$ and~$E\setminus F$
has necessarily zero measure.
\end{proof}

The following lemma is the analogue of Lemma \ref{lemEFdelta} but under the assumption that $E$ is a stable set (not necessarily a minimizer) for the $K$-perimeter.
\begin{lem}\label{lemEFtdelta}
Let $E \subset\R^n$. Assume that $E$ is a stable  set for $P_{K,B_R}$ and that  $F_t = \Psi_t(E)$, where $\Psi_t$ is the integral flow of some vector field $X\in C^2_c(B_R;\R^n)$. Assume moreover that
\begin{equation}\label{near2}
 P_{K,B_R}(F_t)+ P_{K,B_R}(F_{-t}) \le 2P_{K,B_R}(E) +\eta t^2,\quad \mbox{for }t\in(-1,1)
\end{equation}
for some $\eta>0$.

Then, for any $\varepsilon>0$ there exists $t_0>0$ such that for $t\in(-t_0,t_0)$
\[ \min\bigl\{ L_K(F_t\setminus E, E\setminus F_t)\,,\,  L_K(F_{-t}\setminus E, E\setminus F_{-t})\bigr\}\le  (\eta/4 +\varepsilon)t^2.\]
\begin{proof}
Let $C_t=E\cup F_t$ and $D_t= E\cap F_t$. Note that both $C_t$ and $D_t$ coincides with $E$ and $F_t$ outside of $B_R$.  We have
\[
P_{K, B_R}(C_t) +  P_{K, B_R}(D_t) + 2 L_K(F_t\setminus E, E\setminus F_t) = P_{K, B_R}(E) +  P_{K, B_R}(F_t).
\]
and
\[
P_{K, B_R}(C_{-t}) +  P_{K, B_R}(D_{-t}) + 2 L_K(F_{-t}\setminus E, E\setminus F_{-t}) = P_{K, B_R}(E) +  P_{K, B_R}(F_{-t}).
\]

Using \eqref{near2} and the stability of $E$, we deduce
\[
\begin{split}
P_{K, B_R}(C_t) +  &P_{K, B_R}(D_t)+ P_{K, B_R}(C_{-t}) +  P_{K, B_R}(D_{-t})\\
&\hspace{1em} + 2 L_K(F_t\setminus E, E\setminus F_t) +2 L_K(F_{-t}\setminus E, E\setminus F_{-t})\le
\\
&\le 4P_{K,B_R}(E) + \eta t^2
\\
&\le P_{K, B_R}(C_t) +  P_{K, B_R}(D_t)+ P_{K, B_R}(C_{-t}) +  P_{K, B_R}(D_{-t}) +(\eta+4\varepsilon) t^2
\end{split}
\]
for $t\in(-t_0,t_0)$ with $t_0>0$ small enough (depending on $E$ and $X$).
\end{proof}
\end{lem}

We remind that the definition of $E_{R, t}$ depends on the choice of the vector ${\boldsymbol v}\in S^{n-1}$ along which we are translating the set $E$ ---see \eqref{def-psi} and \eqref{def-E+}.
In the sequel, we will use the notion of directional derivative of a $BV$-function in the distributional sense. Let $u\in BV(\Omega)$ and $\boldsymbol v \in S^{n-1}$; we define:
\begin{equation}\label{v-deriv}
|\partial_{\boldsymbol v} u|(\Omega):=\sup\left\{- \int_\Omega u(x) \partial_{\boldsymbol v}\phi(x)dx\; :\; \phi\in C_c^1(\Omega,[-1,1]) \right\},
\end{equation}
and
\begin{equation}\label{v-deriv_pm}
(\partial_{\boldsymbol v}u)_\pm(\Omega) :=\sup\left\{ \mp\int_\Omega u(x) \partial_{\boldsymbol v}\phi(x)\boldsymbol  dx\; :\; \phi\in C_c^1(\Omega,[0,1]) \right\}.
\end{equation}

The following lemma will allow us to obtain geometric informations from the conclusion of Lemma \ref{lemEFtdelta}.

\begin{lem}\label{prop:intermediate}
Let $n\ge 2$, $\eta>0$, $E\subset \R^n$ measurable. Assume that for all ${\boldsymbol v}\in S^{n-1}$, there exists a sequence $t_k \to 0$, $t_k\in (-1,1)$ such that
\begin{equation}\label{assumptionintermediate}
\lim_{k\to \infty}\frac{1}{t_k^2} \,\bigl| \{(E+t_k{\boldsymbol v})\setminus E\}\cap B_1 \bigr|\,\cdot \,
\bigl| \{E \setminus(E+t_k{\boldsymbol v})\}\cap B_1\bigr| \le \frac{\eta}{4}.
\end{equation}

Then,

{\rm (a)} The characteristic function $u =\chi_E$ has finite total variation in $B_1$, that is, $u \in {\rm BV}(B_1)$.

{\rm (b)} For all ${\boldsymbol v}\in S^{n-1}$,  the distributional derivative  $\partial_{\boldsymbol v} u$  is a signed measure on $B_1$ of the form
\[ \partial_{\boldsymbol v} u =  (\partial_{\boldsymbol v} u)_+  -  (\partial_{\boldsymbol v} u)_-\]
with
\[ (\partial_{\boldsymbol v} u)_{\pm} := (-\nu_E\cdot{\boldsymbol v} )_{\pm} H^{n-1}\restrict{\partial^*E \cap B_1} \]
where $\partial^*E$ is the reduced boundary of $E$.

{\rm (c)} For all ${\boldsymbol v}\in S^{n-1}$
\[
\min \left\{ \int_{B_1} (\partial_{\boldsymbol v} u)_+ dx \,,\, \int_{B_1} (\partial_{\boldsymbol v} u)_- dx \right\} \le \frac{\sqrt{\eta}}{2}
\]
and
\[
\max \left\{ \int_{B_1} (\partial_{\boldsymbol v} u)_+ dx \,,\, \int_{B_1} (\partial_{\boldsymbol v} u)_- dx \right\} \le 2|B_1^{(n-1)}| +\frac{\sqrt{\eta}}{2},
\]
where  $|B_1^{(n-1)}|$ denotes the $n-1$-dimensional volume of the ball $B_1\subset \R^{n-1}$.

{\rm (d)} ${\rm Per}_{B_1}(E) =  H^{n-1}(\partial^*E \cap B_1) \leq |S^{n-1}| \left( 1 + \frac{\sqrt{\eta}}{2|B_1^{(n-1)}|}\right)$.
\end{lem}

We next give the
\begin{proof}[Proof of Lemma \ref{prop:intermediate}]
We have
\begin{equation}\label{nameme}
\frac{1}{|t_k|}\min \biggl\{\bigl| \{(E+t_k{\boldsymbol v})\setminus E\}\cap B_1 \bigr|
\,,\,
\bigl| \{E \setminus(E+t_k{\boldsymbol v})\}\cap B_1\bigr|\biggr\} \le \frac{\sqrt{\eta}}{2}.
\end{equation}


Denoting $u=\chi_E$, \eqref{nameme} becomes
\begin{equation}\label{finite_dif_control}
\frac{1}{|t_k|}\min \biggl\{
\int_{B_1} \bigl(u(x-t_k{\boldsymbol v})-u(x)\bigr)_+\,dx
\,,\,
\int_{B_1} \bigl(u(x-t_k{\boldsymbol v})-u(x)\bigr)_-\,dx
\biggr\}
\le \frac{\sqrt{\eta}}{2}.
\end{equation}

Let us now denote the measures
\[\mu_{k,\pm }(dx) = \left(\frac{u(x-t_k{\boldsymbol v})-u(x)}{-t_k}\right)_\pm dx \]
and $\mu_k = \mu_{k,+ }-\mu_{k,- }$.
Note that
\[ \mu_k(B_1)= \int_{B_1}\frac{u(x-t_k{\boldsymbol v})-u(x)}{-t_k}dx=\frac{\int_{B_1+t_k \boldsymbol v} u dx-\int_{B_1} u dx}{t_k}. \]
Hence, since $u$ is a characteristic function,
\[ \bigl| \mu_k(B_1)\bigr|\le  \frac{2 \bigl|(B_1+t_k{\boldsymbol v})\setminus B_1\big|}{|t_k|} \le 2 |B_1^{(n-1)}|,\]
where  $|B_1^{(n-1)}|$ denotes the $(n-1)$-dimensional volume of the ball $B_1^{(n-1)}\subset \R^{n-1}$.

Now, by \eqref{finite_dif_control} we have
\begin{equation}\label{MIN}
 \min\bigl\{  \mu_{k,+}(B_1)\,,\,\mu_{k,-}(B_1) \bigr\} \le \frac{\sqrt{\eta}}{2}.
\end{equation}
But then, since $\mu_k = \mu_{k,+ }-\mu_{k,- }$ we must have
\begin{equation}\label{MAX}
\max\bigl\{  \mu_{k,+}(B_1)\,,\,\mu_{k,-}(B_1) \bigr\} \le 2 |B_1^{(n-1)}| + \frac{\sqrt{\eta}}{2}.
\end{equation}
This implies that both the (nonegative) measures $\mu_{k,+}$, $\mu_{k,-}$ are bounded in $B_1$ independently of $k$. Thus, up to extracting a subsequence, we have  $\mu_{k,+} \rightharpoonup \mu_+$ and $\mu_{k,-} \rightharpoonup \mu_-$ (weak convergence) for some bounded nonnegative measures $\mu_{+}$, $\mu_{-}$.

We have clearly that $\mu_{k} \rightharpoonup  \mu_+ - \mu_-$. Moreover it is immediate to check that, for every $\eta\in C^\infty_c(B_1)$
\[ \int_{B_1}\eta(x) \mu_{k}(dx) = \int_{B_1}\frac{\eta(x+t_k{\boldsymbol v}) -\eta(x)}{-t_k} u(x)\,dx\]
if $t_k$ is smaller than ${\rm dist}\,({\rm spt}\,\eta, \partial B_1)$, where ${\rm spt}\,\eta$ denotes the support of $\eta$ and ${\rm dist}\,(A,B)$ the distance between the sets $A$ and $B$.
It follows that
\begin{equation}\label{lim-t} \lim_{k\to \infty}\int_{B_1}\eta(x) \mu_{k}(dx) = -\int_{B_1} \partial_{\boldsymbol v} \eta(x)u(x) dx \end{equation}
and thus $\mu_+ - \mu_-$ is the distributional derivative of $u$ in the direction $\boldsymbol v$ restricted to $B_1$, which we denote $\partial_{\boldsymbol v} u$.

Moreover from \eqref{MIN} and \eqref{MAX} it follows that
\[ \min\bigl\{  \mu_{+}(B_1)\,,\,\mu_{-}(B_1) \bigr\} \le \frac{\sqrt{\eta}}{2}\]
and
\[ \max\bigl\{  \mu_{+}(B_1)\,,\,\mu_{-}(B_1) \bigr\} \le 2 |B_1^{(n-1)}| + \frac{\sqrt{\eta}}{2}.\]

The above inequalities hold for translations in any direction $\boldsymbol v \in S^{n-1}$, and hence we can choose $\boldsymbol v$ to be the coordinate unit vectors.
We then obtain that there are $n$ signed measures ${\boldsymbol \mu} = (\mu_1,\mu_2, \dots ,\mu_n)$ in $B_1$ such that
\[|\mu_i| (B_1) \le 2 |B_1^{(n-1)}| + \sqrt{2\eta  }, \quad \mbox{ for } i= 1,\dots, n\]

Moreover, since by definition $\mu_i$ is the distributional derivative $\partial_i u$ we have
\[ \sum_i \int_{B_1}   T_i  \mu_i(dx)  = - \int_{B_1} ({\rm div}\,{\boldsymbol T})\, u \,dx\]
for every vector field $\boldsymbol T\in C^1_c(B_1; \R^n)$, where $u = \chi_E$.
This proves (a). Namely,  $u\in {\rm BV}(B_1)$.

We next prove (b). Using that $\partial_{\boldsymbol v} u$ is the distributional derivative of $u=\chi_E$ and applying the divergence theorem for sets of finite perimeter ---see \cite{EvGa}--- we have, for all $\varphi\in C^1_c(B_1)$,
\begin{equation}\label{relation}
\begin{split}
\int_{B_1} \varphi \partial_{\boldsymbol v} u \,dx &=  -\int_{B_1} \partial_{\boldsymbol v} \varphi  u \,dx
\\
&= -\int_{B_1\cap E} {\rm div}(\varphi \boldsymbol v) \varphi
\\
&=  - \int_{\partial^*E}    \varphi (\nu_E \cdot \boldsymbol v) \,d H^{n-1},
\end{split}
\end{equation}
where $\partial^*E$ denotes the reduced boundary of $E$ (in $B_1$).

The identity \eqref{relation} gives the decomposition$\partial_{\boldsymbol v} u =  (\partial_{\boldsymbol v} u)_+  -  (\partial_{\boldsymbol v} u)_-$.
for
\[ (\partial_{\boldsymbol v} u)_{\pm} = -(\nu_E\cdot{\boldsymbol v} )_{\pm} H^{n-1}\restrict{\partial^*E \cap B_1}. \]
Note that the previous decomposition is the Hahn-Jordan decomposition of $ \partial_{\boldsymbol v} u $ since   $(\partial_{\boldsymbol v} u)_+$ and $ (\partial_{\boldsymbol v} u)_-$ are concentrated on disjoint subsets of $\partial^*E$. In particular we deduce that $(\partial_{\boldsymbol v} u)_{\pm} \le \mu_{\pm}$.
Thus (b) and (c) follow. Namely, with the above definitions we have
\[
\min \left\{ \int_{B_1} (\partial_{\boldsymbol v} u)_+ dx \,,\, \int_{B_1} (\partial_{\boldsymbol v} u)_- dx \right\} \le  \frac{\sqrt{\eta}}{2}
\]
and
\[
\max \left\{ \int_{B_1} (\partial_{\boldsymbol v} u)_+ dx \,,\, \int_{B_1} (\partial_{\boldsymbol v} u)_- dx \right\} \le 2|B_1^{(n-1)}| + \frac{\sqrt{\eta}}{2},
\]
where  $|B_1^{(n-1)}|$ denotes the $n-1$-dimensional volume of the ball $B_1\subset \R^{n-1}$.

To prove (d) we integrate with respect to all directions  ${\boldsymbol v}\in S^{n-1}$ the inequality
\[ \int_{\partial^*E}   |\nu_E(x) \cdot \boldsymbol v| \,d H^{n-1}(x) \le 2\,|B_1^{(n-1)}| +\sqrt{\eta},\]
which follows from the previous steps. Using Fubini we find
\[
\begin{split}
 H^{n-1}(\partial^* E) \,2 |B_1^{(n-1)}| &=
\int_{\partial^*E}  \,d H^{n-1}(x) \int_{S^{n-1}}  \,d H^{n-1}( \boldsymbol v)  |\nu_E \cdot \boldsymbol v|\\
&= \int_{S^{n-1}}  \,d H^{n-1}( \boldsymbol v) \int_{\partial^*E}  \,d H^{n-1}(x)   |\nu_E(x) \cdot \boldsymbol v|\\
&\le |S^{n-1}| \,\left( 2|B_1^{(n-1)}| +\sqrt{\eta}\right),
\end{split}
\]
concluding the proof of (d).
\end{proof}

\section{Proof of Theorems \ref{thmstable} and \ref{BV-est}}

In this section we give the proof of our uniform $BV$-estimates.
%

We start with the proof of our general result Theorem \ref{BV-est}.

\begin{proof}[Proof of Theorem \ref{BV-est}]
For the proof we just need to combine Lemma \ref{lem2A} (a),  Lemma \ref{lemEFtdelta}, and Lemma \ref{prop:intermediate}.
More precisely,  by Lemma \ref{lem2A} (a) (applied with $R=4$), we have that
\begin{equation}\label{est-pm}
\begin{split}
P_{K,B_4}(E_{4,t})+P_{K,B_4}(E_{4,-t})-2P_{K,B_4}(E)&\leq 2t^2 P_{K^*,B_4}(E).\\
\end{split}
\end{equation}
Hence, $E$ satisfies the assumption in Lemma \ref{lemEFtdelta} and therefore, for any $\varepsilon >0$ there exists $t_0$ such that for any $t\in (0,t_0)$
\begin{equation}\label{est-eps}\min\{L_K(F_t\setminus E, E\setminus F_t), L_K(F_{-t}\setminus E,E\setminus F_{-t})\}\leq (\eta/4+\varepsilon)t^2,
\end{equation}
with $$\eta=2P_{K^*,B_4}(E).$$
Now using  the assumption \eqref{Kboundedbelow}, namely that $K\ge 1$ in $B_2$ and the definition of $L_K$ we prove that there is a some sequece $t_k\in(-1,1)$ with 
$t_k \downarrow 0$ such that
\[
\lim_{k\to \infty}\frac{1}{t_k^2} \,\bigl| \{(E+t_k{\boldsymbol v})\setminus E\}\cap B_1 \bigr|\,\cdot \,
\bigl| \{E \setminus(E+t_k{\boldsymbol v})\}\cap B_1\bigr| \le \frac{\eta}{4}+\varepsilon,
\]
for all $\epsilon>0$.

After letting $\varepsilon \rightarrow 0$, we apply Lemma \ref{prop:intermediate} and, in particular, from point (d) we deduce that
\[
 {\rm{Per}}_{B_1}(E) \le |S^{n-1}|\left(1+\frac{\sqrt{2P_{K^*,B_4}(E)}}{2|B^{n-1}_1|} \right)\leq \sqrt 2 \,n \sqrt{P_{K^*,B_4}(E)} + |S^{n-1}|,
\]
as wanted.
\end{proof}

In the proof of Theorem \ref{thmstable} we will need the following abstract Lemma. Although this useful abstract statement is due of L. Simon \cite{Simon}, the result was previously well-known in concrete situations, such as in the context of adimensional H\"older norms and their interpolation inequalities. We include its proof here for completeness.

\begin{lem}\label{lem_abstract}
Let $\beta\in \R$ and $C_0>0$. Let $S: \mathcal B \rightarrow [0,+\infty]$, be a nonnegative function defined on the class $\mathcal B$  of open balls $B\subset \R^n$ and satisfying the following subadditivity property
\[  B \subset \bigcup_{j=1}^N B_j \quad \Longrightarrow\quad S(B)\le \sum_{j=1}^N S(B_j). \]
Assume that 
\[ S(B_1)< \infty.\]

There is $\delta = \delta(n,\beta)$ such that if
\begin{equation}\label{hp-lem}
 \rho^\beta S\bigl(B_{\rho/4(z)}\bigr) \le \delta \rho^\beta S\bigl(B_\rho(z)\bigr)+ C_0\quad \mbox{whenever }B_\rho(z)\subset B_1
 \end{equation}
Then
\[ S(B_{1/2}) \le CC_0,\]
where $C= C(n,\beta)$.
\end{lem}
\begin{proof}

Define
\[
Q:= \sup_{B_\rho(z)\subset B_{3/4}} \rho^\beta S\bigl(B_\rho/4(z)\bigr)
\]
We prove first that $Q<\infty$ since $S(B_1)<\infty$. Take $z \in B_{3/4}$. By subadditivity $S\bigl(B_{1/4}(z)\bigr)\le S(B_1)<\infty$. We define
\[ S'(B) = \left(\frac{\rm diam(B)}{2}\right)^{\beta}S(B).\]
Clearly, $S'\bigl(B_{1/4}(z)\bigr) =(1/4)^{\beta} S\bigl(B_{1/4}(z)\bigr) \le 4^{-\beta} S(B_1)$.

On the other hand, by \eqref{hp-lem} we have
\[ S'\bigl(B_{2^{-k-2}}(z)\bigr) \le \delta S'(B_{2^{-k}}(z)) +C_0\]
and thus, if $\delta\le 1/2$, iterating we obtain
\[ S'\bigl(B_{2^{-2k-2}}(z)\bigr) \le S'(B_{1/4}(z)) + C_0 <\infty,\]
for all $k\ge 0$.
But for $r\in (2^{-2(k+1)-2}, 2^{-2k-2})$ we have
\[
\begin{split}
S'\bigl(B_{r}(z)\bigr)
&\le \max\{1,4^{-\beta}\} S'\bigl(B_{2^{-2k-2}}(z)\bigr)
\\
&\le \max\{1,4^{-\beta}\} \left( S'(B_{1/4}(z)) + C_0 \right)
\\
&\le \max\{1,4^{-\beta}\}  \left(  4^{-\beta} S(B_{1}) + C_0 \right)
\end{split}
\]
Thus,
\[ Q \le \max\{1,4^{-\beta}\}  \left(  4^{-\beta} S(B_{1}) + C_0 \right) <\infty.\]

Let us now fix a finite covering  of  $\overline{B_{1/4}}$ by a universal number $M= M(n)$ of balls of radius $1/32$ centered at points of $x_i\in\overline B_{1/4}$, that is
\[ \overline{B_{1/4}} \subset \bigcup_{i=1}^M B_{1/32}(x_i).\]

Now, using the subaditivity of $S$ and assumption \eqref{hp-lem} we have 
\[
\begin{split}
 \rho^\beta S\bigl(B_{\rho/4}(z)\bigr)& \le   8^\beta \sum_{i=1}^M  (\rho/8)^\beta S\bigl(B_{\rho/32}(z+\rho x_i)\bigr)\\
& \le  8^\beta \sum_{i=1}^M \bigl( \delta  (\rho/8)^\beta   S\bigl(B_{\rho/8}(z+\rho x_i)\bigr)+C_0) \\
&= 2^\beta\delta \sum_{i=1}^M  \delta  (4\rho/8)^\beta   S\bigl(B_{\rho/8}(z+\rho x_i)\bigr)+8^\beta MC_0\\
& \le  2^\beta \delta M Q +  8^\beta MC_0, 
\end{split}
\]
where we have used that if $B_\rho(z)\subset B_{3/4}$ then also $B_{4\rho/8}(z+\rho x_i) \subset B_{3\rho/4}(z)\subset B_{3/4}$ and the definition of $Q$.
Thus, taking supremum for all balls $B_\rho(z)\subset B_{3/4}$ in the left hand side we obtain 
\[Q \le 2^\beta\delta M Q + 8^\beta MC_0,\]
and  for $\delta = 2^{\beta-1}/M$ we obtain $Q/2 \le 8^\beta MC_0$, which clearly implies the desired bound on $S(B_{1/2})$.
\end{proof}
We will also use the following standard fact.
\begin{lem}\label{P_s-Per}
Let $E\subset \R^n$ be measurable  and $\Omega \subset \R^n$ be smooth. Let 
\begin{equation}\label{Ptilde}
 \tilde P_{s,B}(E) := \int_{E\cap B}\int_{B \setminus E} \frac{\,dx\,d\bar x}{|x-\bar x|^{n+s}}.
 \end{equation}
Then, 
\begin{equation}\label{P_s-P}
\tilde P_{s,\Omega}(E)\leq C {\rm Per}_{\Omega}(E).
\end{equation}
\end{lem}
\begin{proof}

By \cite{DPV}, Proposition 2.2 and applying the Poincar\'e-Wirtinger inequality we have that
\begin{equation}\label{Sob}
\|u-\overline{u}_\Omega\|_{W^{s,1}(\Omega)}\leq C \|u - \overline{u}_\Omega\|_{W^{1,1}(\Omega)}\leq C\int_{\Omega}|\nabla u| dx,
\end{equation}
where $\overline{u}$ denotes the average of $u$ in $\Omega$.

By the density of $W^{1,1}(\Omega)$ in $BV(\Omega)$ (see Theorem 1.17 in \cite{Giusti}), \eqref{Sob} holds with the right-hand side replaced by $|\nabla u|(\Omega)$. Therefore, for $u=\chi_E$, we have
$$ \tilde P_{s,\Omega}(E) = \frac 1 2 \int_\Omega\int_\Omega \frac{|u(x)-u(\bar x)|}{|x-\bar x|^{n+s}} dxd\bar x \leq \|u-\overline{u}_\Omega\|_{W^{s,1}(\Omega)}\leq C|\nabla u|(\Omega)=  C {\rm Per}_{\Omega}(E),$$
as desired.
\end{proof}

\begin{proof}[Proof of Theorem \ref{thmstable}]
Multiplying the kernel $K\in \mathcal L_2$ by a positive constant, we may assume that $\lambda \geq 2^{n+s}$ and hence $K$ satisfies \eqref{Knonnegative}--\eqref{est-second-deriv} with $K^* = C_1 K$.

Therefore, by Theorem \ref{BV-est}, we immediately deduce that
\begin{equation}\label{key}
 {\rm{Per}}_{B_1}(E) \le C \left( 1+ \sqrt{ P_{K,B_4}(E)}\right) <+\infty,
\end{equation}
where ${\rm{Per}}_{B_1}$ denotes the classical perimeter in $B_1$
and $C$ depends only on $n$, $s$, $\lambda$ and $\Lambda$ ---since $C_1$ depends only on these constants.

Now, since $K\in \mathcal L_2$ and by Lemma \ref{P_s-Per}, we deduce that
\begin{equation}\label{key2}
\begin{split}
P_{K,B_4}(E)
&\le   \Lambda P_{s,B_4}(E)
\\
& \le   \Lambda \int_{E\cap B_4}\int_{B_4 \setminus E} \frac{dx\,d\bar x}{|x-\bar x|^{n+s}}
+ \Lambda \int_{B_4}\int_{\R^n \setminus B_4} \frac{dx\,d\bar x}{|x-\bar x|^{n+s}}
\\
& \le \Lambda \tilde P_{s,B_4}(E) + C
\\
& \le C\left(1+\mbox{Per}_{B_4}(E)\right),
\end{split}
\end{equation}
where $\tilde P_{s,B_4}(E)$ is defined as in \eqref{Ptilde}.

Hence,  \eqref{key}, \eqref{key2} and Young's inequality imply that
\begin{equation}\label{key1}
\begin{split}
 {\rm{Per}}_{B_1}(E)   &\le  C\bigl(1+ \bigl(1+{\rm{Per}}_{B_4}(E)   \bigr)^{1/2}\bigr)\\
  &\le  C(1+\delta^{-1})  + \delta\, {\rm{Per}}_{B_4}(E)  ,
  \end{split}
\end{equation}
for all $\delta>0$, where $C$ depends only on $n$, $s$, $\lambda$, and $\Lambda$.

Next, we observe that, since $E$ is a stable minimal set for $P_{K,B_1}$, with $K\in \mathcal L_2(s,\lambda, \Lambda)$,  given $B_r(z)\subset B_1$ then the rescaled set $E' =(r/4)^{-1}(E-z)$ is a stable minimal set for $P_{K',B_4}$, where
\[   K'(y) : = {(r/4)}^{n+s} K(ry/4) \quad \mbox{belongs again to }\mathcal L_2(s,\lambda, \Lambda).\]
Thus, rescaling the estimates \eqref{key1} applied to $E'$ we obtain, for $E$,
\begin{equation}
 r^{1-n} \,{\rm{Per}}_{B_{r/4}(z)}(E)  \le  C(n, s, \lambda, \Lambda, \delta) + \delta \,r^{1-n}\, {\rm{Per}}_{B_{r}(z)}(E).
\end{equation}

Therefore, considering the subadditive function on the class of balls
\[S(B):= {{\rm Per}_B}(E), \]
and taking $\beta:=1-n$ , and  $\delta = \delta(n,\beta)$ given by Lemma \ref{lem_abstract} we find that
\[ S(B_{1/2}) \le  C(n,s,\lambda, \Lambda).\]
since $S(B_1)<+\infty$ by \eqref{key} --- note since $E$ is a stable minimal set in $B_4$ by definition we have $P_{K,B_4}(E)<+\infty$.

Thus, we have shown that
\[
 {\rm{Per}}_{B_{1/4}}(E) \leq C(n,s,\lambda,\Lambda),
\]
where $C(n,s,\lambda, \Lambda)$ is a universal constant depending only on $n,s,\lambda, \Lambda$.

By scaling and using  a standard covering argument, we obtain
\begin{equation}\label{bound-P}
{\rm Per}_{B_{1}}(E)  \le C(n,s,\lambda, \Lambda),
\end{equation}
which finishes the proof.
\end{proof}
\begin{proof}[Proof of Corollary \ref{thmstable1}]
We combine the universal perimeter estimate in $B_1$ of Theorem \ref{thmstable} ---see \eqref{bound-P}--- with the``interpolation inequality'' $ P_{K,B_1}(E)  \le C\bigl(1+{\rm Per}_{B_{1}}(E)\bigr)$,  shown in \eqref{key2}, to obtain $P_{K,B_1}(E) \le C$. The estimate for the $K$-perimeter in $B_R$ then follows using the scaling invariance of the class $\mathcal L_2(s,\lambda,\Lambda)$.
\end{proof}

\section{Proof of Theorems \ref{aniso} and \ref{flatness}}
Before giving the proofs of Theorems \ref{aniso} and \ref{flatness}, we give some preliminary lemmas.
We start with the following easy fact, that we state explicitly since we will use it several times later on.
\begin{rem}\label{odd} Let $\Phi$ be a continuous and odd function defined on the $m$-dimensional sphere $S^m$, with $m\geq 1$.

Then, there exists $\boldsymbol v^*\in S^m$ such that $\Phi(\boldsymbol v^*)=0$.

The proof of this fact is obvious since $S^m$ is connected when $m\geq 1$.
\end{rem}
\begin{lem}\label{lem-max}
Suppose that $\Phi_+$ and $\Phi_-$ are two continuous functions defined on $S^{n-1}$, which satisfy
\begin{equation}\label{hp1-lem-max}
\Phi_+(-\boldsymbol v)=\Phi_-(\boldsymbol v) \quad \mbox{for any}\;\;\boldsymbol v \in S^{n-1}.
\end{equation}
Assume moreover that there exists  $\mu>0$ such that for any $\boldsymbol v \in S^{n-1}$
\begin{equation}\label{hp2-lem-max}
\min\{\Phi_+(\boldsymbol v), \Phi_-(\boldsymbol v)\}\leq \mu.
\end{equation}

Then, after a rotation of coordinates, we have that
\begin{equation}\label{max}
\max\{\Phi_+(\boldsymbol e_i),\Phi_-(\boldsymbol e_i)\}\leq \mu \quad \mbox{for  }1\le i\le n-1,
\end{equation}
where $\boldsymbol e_i$ denote the standard basis of $\R^n$,
\end{lem}
\begin{proof}
For $\boldsymbol v \in S^{n-1}$, we consider the function
$$\Phi(\boldsymbol v)=\Phi_+(\boldsymbol v)-\Phi_-(\boldsymbol v).$$
Using \eqref{hp1-lem-max}, it is easy to verify that $\Phi$  is odd and hence, using Remark \ref{odd},  there exists a vector $\boldsymbol v^*_1 \in S^{n-1}$ for which
$$\Phi(\boldsymbol v^*_1)=\Phi_+(\boldsymbol v^*_1)-\Phi_-(\boldsymbol v^*_1)=0.$$
This clearly implies that
$$
\Phi_+(\boldsymbol v^*_1)=\Phi_-(\boldsymbol v^*_1)=\min\{\Phi_+(\boldsymbol v^*_1),\Phi_-(\boldsymbol v^*_1)\}=\max\{\Phi_+(\boldsymbol v^*_1),\Phi_-(\boldsymbol v^*_1)\}.$$
Hence, by \eqref{hp2-lem-max}, we deduce that
\begin{equation}\label{max-Phi}\
\max\{\Phi_+(\boldsymbol v^*_1),\Phi_-(\boldsymbol v^*_1)\}\leq \mu.
\end{equation}
Now we define $\Phi_2$ to be $\Phi$ restricted to the $(n-2)$-dimensional sphere given by $S^{n-1}\cap \boldsymbol (v^*_1)^\perp$.
By Remark \ref{odd} applied now to $\Phi_2$, there exists a vector $\boldsymbol v^*_2 \in S^{n-1}\cap (v^*_1)^\perp$ for which
\eqref{max-Phi} holds (with $\boldsymbol v^*_1$ replaced by $\boldsymbol v^*_2$). We can iterate this procedure $(n-1)$  times: at each step
we apply Remark \ref{odd} to the function $\Phi_i$, that is the restriction of $\Phi$ to the $(n-i)$-dimensional sphere $S^{n-1}\cap (\boldsymbol v_1^*)^\perp\cap\dots\cap (\boldsymbol v_{i-1}^*)^\perp$.
In this way we get $(n-1)$ vectors $\boldsymbol v^*_1,\dots,\boldsymbol v^*_{n-1}$ which are orthonormal and for which \eqref{max-Phi} holds (with $\boldsymbol v^*_1$ replaced by $\boldsymbol v^*_i$, $1\le i\le n-1$).
After some orthogonal transformation, we may assume $\boldsymbol v^*_i=\boldsymbol e_i$, for $i=1,\dots, n-1$.
\end{proof}
To prove Theorems \ref{aniso} and \ref{flatness} we will use an argument with some flavor of ``integral geometry''. The use of a integral geometry approach for the
study of anisotropic nonlocal perimeter functionals turns out to be  useful also in the recent paper of Ludwig \cite{L}.

Let us introduce some notation. In the the sequel $L\subset\R^n$ denotes a linear subspace with dimension $m$ with $1\le m\le n-1$.
We let  $\{\boldsymbol v_i\}_{1\le i \le m}$ be a orthonormal basis of $L$ and denote
\[ L^\perp = \{y \ :\  \boldsymbol v_i\cdot y = 0 \mbox{ for all }1\le i\le m\}.\]

Let $\Omega\subset \R^n$ be a bounded open set. Given a set $E$ with finite perimeter in $\Omega$, let $u = \chi_E$. Note that the distributional gradient $\nabla u$ is a vector valued measure in $B_1$. We will denote $\nabla_L u$ the projected (vector valued) measure
\[ \nabla_L u = \sum_{i=1}^m (\nabla u\cdot\boldsymbol v_i) \boldsymbol v_i.\]

For each (almost every) $y\in {L}^\perp$ we denote $I_{E,\Omega}(L,y)$ the total variation of $u =  \chi_E$ restricted to $(y+L)\cap \Omega$. That is, we define
\begin{equation}\label{VLy}
I_{E,\Omega}(L,y) := \sup \biggl\{- \int_{(y+L)\cap \Omega} u(z) {\rm div} \, \phi(z) \,dH^{m}(z) \ : \  \phi\in C^1_c \bigl((y+L)\cap \Omega; L\cap B_1 \bigr)    \biggr\}.
\end{equation}

Sometimes, when $E$ and $\Omega$ are fixed and there is no
misunderstanding, for the sake of simplicity we will also use the notation 
$$I(L,y):=I_{E,\Omega}(L,y).$$

When $m=1$ and $L=\R \boldsymbol v$ for some $ \boldsymbol v\in S^{n-1}$ we will also denote $I(L,y)$ as $I(\boldsymbol v,y)$.
In the case $m=1$ we define also $I(\boldsymbol v,y)_+$ and $I(\boldsymbol v,y)_-$ respectively as
\begin{equation}\label{VLypm}
I(\boldsymbol v,y)_\pm := \sup \biggl\{ \mp \int_{(y+\R \boldsymbol v)\cap \Omega} u(z)  \phi'(z) \,dH^{1}(z) \ :\
  \phi\in C^1_c \bigl((y+\R \boldsymbol v)\cap \Omega; [0,1] \bigr)\biggr\},
\end{equation}
where $\phi' = \partial_{\boldsymbol v}\phi$ denotes the tangential derivative along the (oriented) line $y+\R \boldsymbol v$.
This auxiliary function~$I(\boldsymbol v,y)_\pm$ is useful to detect
the monotonicity of~$\chi_E$, as pointed out in the following result:

\begin{lem}\label{L:I:Z}
Let~$E$ be a set of finite perimeter in a convex open set $\Omega$,
$\boldsymbol v\in S^{n-1}$ and~$y\in {\boldsymbol v}^\perp$.
Then:
\begin{itemize}
\item[{(i)}] If~$I(\boldsymbol v,y)_+ =0$, then~$\chi_E$ restricted to~$(y+\R
{\boldsymbol v})\cap\Omega$ is nonincreasing;
\item[{(ii)}] If~$I(\boldsymbol v,y)_- =0$, then~$\chi_E$ restricted to~$(y+\R
{\boldsymbol v})\cap\Omega$ is nondecreasing;
\item[{(iii)}] If~$I(\boldsymbol v,y)=0$, then~$(y+\R
{\boldsymbol v})\cap\Omega$ is contained either in~$E$ or in ${\mathcal{C}}E$. 
\end{itemize}
\end{lem}

\begin{proof} To prove~(i), we denote  $(a,b)\subset \R$ the open interval 
$ \{t\in\R \ : \ y+t{\boldsymbol v}\in \Omega \}$ ---here we use the convexity of $\Omega$.
Let us define~$\tilde u(t):=
\chi_E(y+t{\boldsymbol v})$ and we remark that~$\tilde u$
is of bounded variation in~$[a,b]$ ---see e.g. Corollary~6.9
of~\cite{ALBE} or Theorem 2 in Section 5.10.2 of~\cite{EvGa}. 
Then, given any~$\phi\in C^1_c\bigl( (y+\R{\boldsymbol v})\cap \Omega;\;[0,1]\bigr)$,
we define~$\tilde\phi(t):=\varphi(y+t{\boldsymbol v})$ and
we use~\eqref{VLypm} to find that
$$ 0=I(\boldsymbol v,y)_+ \ge- \int_a^b u(y+t{\boldsymbol v})\,
\phi'(y+t{\boldsymbol v})\,dt =
-\int_a^b \tilde u(t)\,\tilde\phi'(t)\,dt$$
for all  $\tilde\phi \in C^1_c\bigl( (a,b)\cap \Omega;\;[0,1]\bigr)$.
As a consequence (see e.g. Corollary~9.91 in~\cite{IMAGE}),
we have that~$\tilde u$ is nonincreasing, which is~(i).

The proof of~(ii) is analogous. Now we prove~(iii). By taking
$\phi$ identically zero in~\eqref{VLypm}, we see 
that~$I(\boldsymbol v,y)_\pm\ge0$.
Therefore, if~$I(\boldsymbol v,y)=0$, then~$I(\boldsymbol v,y)_+ =
I(\boldsymbol v,y)_-=0$, and thus we can use~(i) and~(ii)
to deduce that~$\chi_E$ restricted to~$y+\R
{\boldsymbol v}$ is constant, which gives~(iii).
\end{proof}

The following proposition gives equivalent formulas to compute the total variation of the projection of onto some linear subspace $L$ of the measure $\nabla u$, $u$ being the characteristic function of a set of finite perimeter.
\begin{prop}\label{prop-fubini-typeV}
Let $\Omega\subset \R^n$ be a bounded open set,  $E$ be a set of finite perimeter in $\Omega$, and  $u = \chi_E$. Let $L\subset\R^n$  be linear subspace with dimension $m$ with $1\le m\le n-1$.
We let  $\{\boldsymbol v_i\}_{1\le i \le m}$ be a orthonormal basis of $L$.

Then,  $I_{E,\Omega}(L,y) \ge0 $ is  measurable in the variable $y\in L^\perp$ and the following identities  hold
\begin{equation}\label{totvar_L}
\begin{split}
|\nabla_L u| (\Omega)
:&=  \sup \left\{  -\int_{\Omega} u(x)  {\rm div} \phi(x)   \,dx \ :\ \phi\in C^1_c (\Omega;L\cap B_1)   \right\} \\
& =  \int_{\partial^*E\cap \Omega} \sqrt{ \sum_{i=1}^m \bigl({\boldsymbol v_i}\cdot \nu_E(x)\bigr)^2 } \,dH^{n-1}(x)\\
&=  \int_{L^\perp} I_{E,\Omega}(L,y) \,dH^{n-m}(y)
\end{split}
\end{equation}

Moreover if $m=1$ and $L=\R\boldsymbol v$ then
\begin{equation}\label{vartot_v}
|\partial_{\boldsymbol v} u| (\Omega)
=  \int_{\partial^*E\cap \Omega} \bigl| {\boldsymbol v}\cdot \nu_E(x) \bigr|dH^{n-1}(x) =  \int_{\boldsymbol v^\perp } I_{E,\Omega}(\boldsymbol v,y) \,dH^{n-1}(y),
\end{equation}
\begin{equation}\label{varpm_v}
(\partial_{\boldsymbol v} u)_{\pm} (\Omega)
=  \int_{\partial^*E\cap \Omega} \bigl(- {\boldsymbol v}\cdot \nu_E(x) \bigr)_{\pm}dH^{n-1}(x) 
=  \int_{\boldsymbol v^\perp } I_{E,\Omega}(\boldsymbol v,y)_\pm  \,dH^{n-1}(y),
\end{equation}
and for a.e. $y\in \boldsymbol v^\perp$ we have
\begin{equation} \label{Icount}
I_{E,\Omega}(\boldsymbol v,y) =  H^0 \bigl(  \partial^* E \cap \Omega\cap (y+\R\boldsymbol v)  \bigr),
\end{equation}
\begin{equation} \label{Ipmcount}
I_{E,\Omega}(\boldsymbol v,y)_\pm = H^0 \bigl( 
\bigl\{ x\in \partial^* E \cap \Omega
\cap (y+\R\boldsymbol v) \ : \  
\mp\boldsymbol v\cdot \nu_E(x) > 0  \bigr\}\bigr).
\end{equation}
\end{prop}

The proof of Proposition \ref{prop-fubini-typeV} relies on standard results from the theory of sets of finite perimeter and functions of bounded variation (see \cite{Mag, EvGa}), and will be sketched in the Appendix.
Note that if $\partial E$ has smooth boundary in $B_1$ then the proof of Proposition \ref{prop-fubini-typeV} is rather elementary.
For related results for $m=1$ in the context of integral geometry formulae for sets of finite perimeter see also \cite[Section 1.1]{L} and \cite[Theorem 1]{W}.

The well-known Cauchy-Crofton formula (and indeed a generalized
version of it) can be obtained as a corollary of the previous proposition with $m=1$, as pointed out by the next result:

\begin{cor}\label{CCformula}
Let $E$ be a set of finite perimeter in $B_1$ and ${\boldsymbol v}\in S^{n-1}$. Let ${\boldsymbol v}^\perp$ denote the hyperplane $\{y \ :\  \boldsymbol v\cdot y = 0\}$.

Then
\[
{\rm Per}_{\Omega}(E) = c \int_{S^{n-1}}dH^{n-1}({\boldsymbol v}) \int_{\boldsymbol v^\perp} dH^{n-1}(y) \,H^0 \bigl(  \partial^* E \cap \Omega \cap (y+\R\boldsymbol v)  \bigr)
\]
where $H^0 \bigl(  \partial^* E \cap \Omega \cap (y+\R\boldsymbol v) \bigr)$ counts the number of intersections inside $\Omega$ of the line $y+\R\boldsymbol v$ with the reduced boundary of $E$.
The constant $c=c(n)$ is given by
\[ c = \left( \int_{S^{n-1}}|\boldsymbol v\cdot \boldsymbol w| dH^{n-1}({\boldsymbol v}) \right)^{-1}\]
where $\boldsymbol w\in S^{n-1}$ is any fixed unit vector ---this value does not depend on $\boldsymbol w$.
\end{cor}

\begin{proof}
Using \eqref{vartot_v} and \eqref{Icount}, we have
$$\int_{\partial^*E\cap\Omega}|\boldsymbol v \cdot \nu_E(x)|dH^{n-1}(x)=\int_{\boldsymbol v^\perp}H^0\big(\partial^*E\cap\Omega \cap(y+\R\boldsymbol v)\big)dH^{n-1}(y).$$
The corollary follows integrating with respect to $\boldsymbol v\in S^{n-1}$.
\end{proof}

The following observation will be crucial in the proof of our Theorems \ref{aniso} and \ref{flatness}.
\begin{rem}\label{I_pm}
When $m=1$, for a.e.  $y\in {\boldsymbol v^\perp}$,
\[I(\boldsymbol v,y),\  I(\boldsymbol v,y)_+\mbox{ and } I(\boldsymbol v,y)_- \mbox{ are nonnegative integers}.\]
Indeed this follows from \eqref{Icount} and \eqref{Ipmcount} since $H^0$ is the counting measure.
\end{rem}

In the rest of this section we will consider the functions
\begin{equation}\label{defPhi}
\Phi_+({\boldsymbol v}) := ( \partial_{\boldsymbol v} u)_+(B_1)\quad \mbox{and}\quad \Phi_-({\boldsymbol v}): = ( \partial_{\boldsymbol v} u)_-(B_1),
\end{equation}
where  $u= \chi_E$ is the characteristic function of a set $E$ of finite perimeter in $B_1$.
By \eqref{varpm_v}, we have
\begin{equation}\label{phipm}
\Phi_\pm(\boldsymbol v)=
\int_{\boldsymbol v^\perp}I_{E,B_1}(\boldsymbol v, y)_\pm dH^{n-1}(y).
\end{equation}
With this observation, we can reformulate Lemma~\ref{L:I:Z} in this way:

\begin{lem}\label{L:I:Z:2}
Let $E$ be a set of finite perimeter in $B_1$,
$\boldsymbol v\in S^{n-1}$ and~$\mu\ge0$.
Then:
\begin{itemize}
\item[{(i)}] If~$\Phi_+ (\boldsymbol v)\le\mu$ (resp.~$
\Phi_- (\boldsymbol v)\le\mu$), then
there exists~${\mathcal{B}}\subseteq {\boldsymbol v}^\perp$ with $H^{n-1}(\mathcal B)\leq \mu$ and such that
for any~$y\in {\boldsymbol v}^\perp\setminus {\mathcal{B}}$ we have that~$\chi_E$ restricted to~$(y+\R
{\boldsymbol v})\cap B_1$ is nonincreasing (resp.
nondecreasing);
\item[{(ii)}] If~$\max\big\{
\Phi_+ (\boldsymbol v),\;\Phi_- (\boldsymbol v)\big\}\le\mu$, then
there exists~${\mathcal{B}}\subseteq {\boldsymbol v}^\perp$ with $H^{n-1}(\mathcal B)\leq \mu$ and such that
for any~$y\in {\boldsymbol v}^\perp\setminus {\mathcal{B}}$ we have
that~$(y+\R
{\boldsymbol v})\cap B_1$ is contained either in~$E$ or in ${\mathcal{C}}E$. 
\end{itemize}
\end{lem}

\begin{proof} Since~(ii) follows from~(i), we focus on the proof of~(i)
and we suppose that~$\Phi_+ (\boldsymbol v)\le\mu$ 
(the case~$\Phi_- (\boldsymbol v)\le\mu$ is analogous).
We set
$${\mathcal{B}}:=\{ y\in {\boldsymbol v}^\perp\;:\; 
I_{E,B_1}(\boldsymbol v, y)_+ \ne0\}.$$
By Remark~\ref{I_pm}, we have that
$${\mathcal{B}}=\{ y\in {\boldsymbol v}^\perp\;:\; 
I_{E,B_1}(\boldsymbol v, y)_+ \ge1\}$$
and therefore, by~\eqref{phipm},
$$ \mu\ge \Phi_- (\boldsymbol v)=
\int_{ {\mathcal{B}} }
I_{E,B_1}(\boldsymbol v, y)_+ dH^{n-1}(y)
\ge H^{n-1}( {\mathcal{B}} ),$$
which is the desired estimate on~${\mathcal{B}}$.

Notice that, by construction, if~$y\in
{\boldsymbol v}^\perp\setminus {\mathcal{B}}$, then~$
I_{E,B_1}(\boldsymbol v, y)_+ =0$, and so
Lemma~\ref{L:I:Z} gives that~$\chi_E$ restricted to~$(y+\R
{\boldsymbol v})\cap\Omega$ is nonincreasing, as desired.
\end{proof} 

With this, we obtain the following flatness result:

\begin{lem}\label{geometric-info}
Let $E$ be a set of finite perimeter in $B_1$, $u= \chi_E$,  $\boldsymbol v\in S^{n-1}$ and $\Phi_\pm$ be as in \eqref{defPhi}.

Suppose that for all $\boldsymbol v\in S^{n-1}$,
\begin{equation}\label{hp-Phi} \min\bigl\{ \Phi_+(\boldsymbol v), \Phi_-(\boldsymbol v)\bigr\}  \leq \mu,\end{equation}
 for some $\mu>0$. Then, after some rotation the set $E$ satisfies {\rm ({\bf F1})},   {\rm ({\bf F2})} , and  {\rm ({\bf F3})} on page \pageref{FFPA},
with
 $$\varepsilon = C(n) \mu,$$ 
 where $C(n)$ is a constant depending only on the dimension.
 \end{lem}

\begin{proof}

We first observe that, since $E$ has finite perimeter in $B_1$, for $u=\chi_E$, $\nabla u$ is a vector valued measure and 
\[ \Phi_\pm(\boldsymbol v) = (\partial_{\boldsymbol v}u)_{\pm}(B_1)=  (\nabla u \cdot \boldsymbol v)_{\pm}(B_1).\]
Then, 
\[\Phi_\pm(-\boldsymbol v) = \Phi_\mp(\boldsymbol v). \]
In addition, we have
\[
|\Phi_+(\boldsymbol v)-\Phi_+(\boldsymbol w)|\le  | \boldsymbol v-\boldsymbol w| \,|\nabla u|(B_1)\]
and same holds for $\Phi_-$. Hence, in particular, $\Phi_+$ and $\Phi_-$ are continuous functions on $S^{n-1}$ satisfying the assumptions of Lemma  \ref{lem-max}.

Therefore, after some rotation we have
\begin{equation}\label{max-phi} 
\max\bigl\{ \Phi_+(\boldsymbol e_i), \Phi_-(\boldsymbol{e_i})\bigr\}  \leq \mu, \quad \mbox{for}\;\;{1\le i\le n-1}
\end{equation}
In addition, by \eqref{hp-Phi}, and possibly changing $\boldsymbol e_n$ by $-\boldsymbol e_n$, we may assume that
\begin{equation} \label{initial-point2}
\Phi_+(\boldsymbol e_n)\leq \mu.
\end{equation}
Using~\eqref{max-phi} and Lemma~\ref{L:I:Z:2}
we conclude that, for any~$i\in\{1,\dots,n-1\}$, there exists~${\mathcal{B}}_i
\subseteq {\boldsymbol e}_i^\perp$, with~$H^{n-1}({\mathcal{B}}_i)\le\mu$,
and such that
for any~$y\in {\boldsymbol e}^\perp_i\setminus {\mathcal{B}}_i$ we have 
that
\begin{equation}\label{LA-01}
{\mbox{
$(y+\R
{\boldsymbol e}_i)\cap B_1$ is contained either in~$E$ or in ${\mathcal{C}}E$
}}.
\end{equation}
Similarly, by~\eqref{initial-point2} and
Lemma~\ref{L:I:Z:2},
we see that
there exists~${\mathcal{B}}_n
\subseteq {\boldsymbol e}_n^\perp$, with~$H^{n-1}({\mathcal{B}}_n)\le\mu$,
and such that
for any~$y\in {\boldsymbol e}^\perp_n\setminus {\mathcal{B}}_n$ we have 
that
\begin{equation}\label{LA-02}
{\mbox{$\chi_E$ restricted to~$(y+\R{\boldsymbol e}_n)\cap B_1$
is nonincreasing.}}
\end{equation}
Notice that~\eqref{LA-02} implies~{\rm ({\bf F2})}.
Now we complete the proof of the desired result in three steps:
first, we establish {\rm ({\bf F1})}
in the two-dimensional case, then in the three-dimensional case,
and finally we prove {\rm ({\bf F3})}.\medskip

\textit{Step 1.} Let us show that {\rm ({\bf F1})}
holds for $\varepsilon = 2\mu$ first in dimension $n=2$.
Let us assume that $\mu<2$ since otherwise $2\mu>\pi 1^2= |B_1|$ and there is nothing to prove.

By~\eqref{LA-01},
for any~$t$ outside the small set~$\mathcal B_1$,
\begin{equation}\label{JAK:9}
{\mbox{the segment $\{x_2 = t\} \cap B_1$ is either contained in $E$ or in $\mathcal CE$;}}\end{equation}
here, we are identifying points $y\in
{\boldsymbol e}_1^\perp$ and points $t\in(-1,1)$ via $y=(0,t)$.

Therefore, we can define~$\mathcal G_{E}$ (resp.,~$\mathcal G_{
{\mathcal C}E}$) as the family
of~$t\in (-1,1)$ for which~$
\{x_2 = t\} \cap B_1$ is contained in $E$ (resp., in~${
{\mathcal C}E}$), and then~\eqref{JAK:9} says that
\[  (-1,1) =  \mathcal G_E \cup \mathcal G_{\mathcal C E} \cup \mathcal B_1.\]
The fact that $\chi_E$ is nonincreasing along the vertical direction for
a (nonvoid)
set of vertical segments 
(as warranted by~\eqref{LA-02}),
implies that the sets $\mathcal G_E$ and $\mathcal G_{\mathcal C E}$ are ordered with respect to the vertical direction. More precisely,  there exist $t_*, t^* \in[-1,1]$, such that
\begin{equation*}
{\rm ess\,sup}  \,\mathcal G_E = t_* \,\le\, t^*  ={\rm ess\,inf}\,  \mathcal G_{\mathcal CE}.
\end{equation*}
This implies that, for all $t\in[t_*,t^*]$,
 \[
 \begin{split}
|(E \setminus \{x_2\le t\})\cap B_1 | + |(\{x_2\le t\}  \setminus E)\cap B_1|
& \le \bigl| \{x_2 \in \mathcal B_1\} \cap B_1 \bigr|
\\
& \le 2\bigl| \mathcal B_1 \bigr| \le  2\mu
\end{split}
\]
and thus {\rm ({\bf F1})} follows.

In dimension $n=2$ we can obtain an even stronger information since
\[ E\cap \{x_1 \notin \mathcal B_2\} \cap B_1 \supset  \{x_2\le t_*\} \cap\{x_1 \notin \mathcal B_2\} \cap B_1\]
and
\[ \mathcal C E\cap \{x_1 \notin \mathcal B_2\} \cap B_1 \supset  \{x_2\ge t^*\} \cap\{x_1 \notin \mathcal B_2\} \cap B_1.\]
Therefore, there exits $g:B_1^{(n-1)}\rightarrow [-1,1]$ such that
\[ E\cap \{x_1 \notin \mathcal B_2\} \cap B_1= \{x_2\le g(x_1)\}\cap \{x_1 \notin \mathcal B_2\} \]
with the oscillation of $g$ bounded by $(t^*-t_*) \le H^1(\mathcal B) \le \mu$ (see Figure 5).

\begin{figure}[h]\label{2dimpicture}
	\centering	
	\resizebox{15cm}{!}{
	
		\begin{tikzpicture}[scale=1]
		
			\node (myfirstpic) at (0,0) {\includegraphics[scale=0.6]{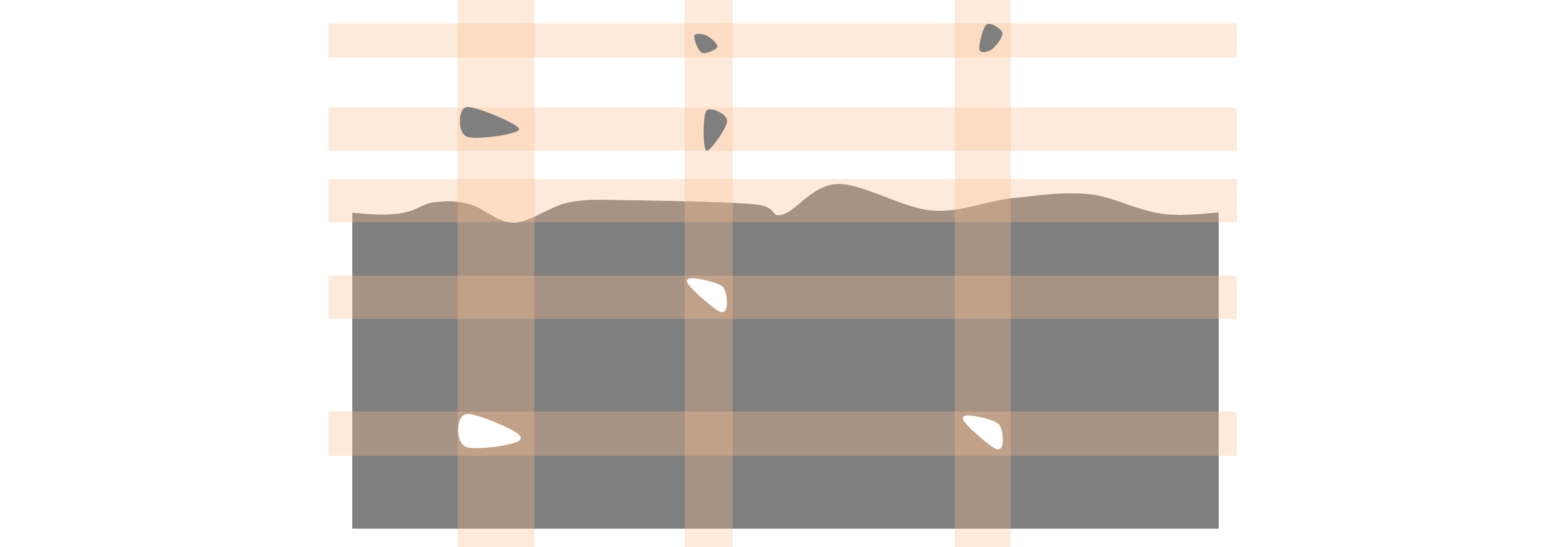}};
			
			
		\end{tikzpicture}
		
	}
	\caption{The two-dimensional picture.}
\end{figure}


\textit{Step 2.} Let us show that~{\rm ({\bf F1})}
holds for $\varepsilon = C(n)\mu$ in dimensions  $n\ge3$.
For this, we define $L= \boldsymbol e_n^\perp$ and we use 
\eqref{totvar_L} and \eqref{vartot_v}
in Proposition \ref{prop-fubini-typeV} to estimate
\begin{equation}\label{totvar_L2-PRE}
\begin{split}
|\nabla_L u| (B_1)
& =  \int_{\partial^*E\cap B_1} \sqrt{ \sum_{i=1}^{n-1} \bigl({\boldsymbol e_i}\cdot \nu_E(x)\bigr)^2 } dH^{n-1}(x)\\
& \le \frac{1}{\sqrt{n-1}} \sum_{i=1}^{n-1} \, \int_{\partial^*E\cap B_1} \bigl|{\boldsymbol e_i}\cdot \nu_E(x)\bigr|  dH^{n-1}(x)\\
& =  \frac{1}{\sqrt{n-1}} \sum_{i=1}^{n-1}  \,|\partial_{\boldsymbol e_i}u| (B_1).
\end{split}
\end{equation}
Now we observe that, by \eqref{defPhi} and \eqref{max-phi},
$$ |\partial_{\boldsymbol e_i}u|(B_1)=
(\partial_{\boldsymbol e_i}u)_+(B_1) +
(\partial_{\boldsymbol e_i}u)_-(B_1)
= \Phi_+({\boldsymbol e_i}) +\Phi_-({\boldsymbol e_i})
\le 2\mu, 
$$
which, together with \eqref{totvar_L2-PRE}, gives that
\begin{equation}\label{totvar_L2}
|\nabla_L u| (B_1) \le 2\sqrt{n-1} \,\mu.
\end{equation}
%
Moreover, we note that there exists a small constant $\bar \mu = \bar\mu(n) >0$  ---depending only on $n$--- such that for $\mu\in (0, \bar\mu)$ and 
$r=\mu^{\frac{1}{n+1}}$ we have
 \begin{equation}\label{proportion}
 \mu \le c \mu^{\frac{n-1}{n+1}} \le \frac{1}{4} H^{n-1}\bigl( B_r^{(n-1)} \bigr).
\end{equation}

We now use that $I_{E,B_1}(L,y)$ is the relative perimeter of $(y+L)\cap E$ in 
the $(n-1)$-dimensional ball $(y+L)\cap B_1$ ---recall \eqref{VLy}.
Thus, using  the relative isoperimetric inequality at each horizontal slice $B_1\cap \{x_n =t\}$ we find that
\begin{equation}\label{est-min}
\begin{split}
&\min\biggl\{ H^{n-1}\bigl(E \cap B_1\cap \{x_n =t\}\bigr),  \,H^{n-1}\bigl(\mathcal C E \cap B_1\cap \{x_n =t\}\bigr) \biggr\}\,
\\
&\hspace{40mm}\le C \min \bigl\{ 1 ,  I_{E,B_1}(L,y)^{\frac{n-1}{n-2}}\bigr\}
\le  \bar C I_{E,B_1}(L,y),
\end{split}
\end{equation}
where $\bar C >0$ is a suitable constant (depending only on $n$)
and the last coordinate of~$y$ equals to~$t$.

Let us define the ``horizontal bad set'' as
$$
\mathcal B := \mathcal B'\cup \mathcal B'',$$
where
\begin{eqnarray*}
&& \mathcal B' :=
\bigl\{ t\in(-1,1)\ :\ |t|\ge \sqrt{1-r^2}\bigr\}  \\
{\mbox{and }}&&
\mathcal B'':=\bigl\{t\in(-1,1)\ :\  |t| <\sqrt{1-r^2} \mbox{ and } \bar CI_{E,B_1}(L,y)>\bar\mu \bigr\}.
\end{eqnarray*}
We also define~$\mathcal G_E$ as the family of~$t\in(-1,1)$
for which $|t|<\sqrt{1-r^2}$ and
\begin{equation}\label{GE}
 H^{n-1}\bigl( \{x_n =t\}\cap B_1 \cap {\mathcal{C}}E \bigr)\le \mu
.\end{equation}
Similarly, we define~$\mathcal G_{\mathcal C E}$
as the family of~$t\in(-1,1)$
for which $|t|<\sqrt{1-r^2}$ and
\begin{equation}\label{GCE}
 H^{n-1}\bigl( \{x_n =t\}\cap B_1 \cap  E \bigr)\le \mu.
\end{equation}
By\eqref{est-min}, it follows that
\begin{equation}\label{CO:11} 
(-1,1)\setminus \mathcal B \subseteq  \mathcal G_E \cup \mathcal G_{\mathcal C E}. \end{equation}
In addition, by \eqref{totvar_L}  and \eqref{totvar_L2},
$$ 2\sqrt{n-1} \,\mu\ge
|\nabla_L u| (B_1) \ge
\int_{\mathcal B''} I_{E,\Omega}(L,y) \,dH^{1}(y)
\ge \frac{ \bar\mu}{\bar C}\,H^{1}({\mathcal B''}).$$
Therefore~$H^{1}({\mathcal B''})\le C_0\,\mu$ and then
\begin{equation}\label{JU:0}
H^{n}\bigr( \{ (x',t)\in B_1 \ :\  t\in \mathcal B''\} \bigl)\le
C_1\, H^{1}({\mathcal B''})\le C_2\,\mu,
\end{equation}
for some constants~$C_0$, $C_1$, $C_2>0$.

Furthermore, if~$(x',t)\in B_1$ and~$t\in \mathcal B'$, then
$$ |x'|^2=|x'|^2+t^2 -t^2\le 1-(1-r^2)=r^2,$$
which implies that $x'\in B_r^{(n-1)}$, and so
that
$$ H^{n}\bigr( \{ (x',t)\in B_1 \ :\  t\in \mathcal B'\} \bigl)\le
C_3\, r^{n-1} \,\big( 1-\sqrt{1-r^2}\big)\le C_4\,r^{n+1}=C_4\,\mu,$$
for some~$C_3$, $C_4>0$.
This and~\eqref{JU:0} give that
\[  H^{n}\bigr( \{ (x',t)\in B_1 \ :\  t\in \mathcal B\} \bigl) \,\le C_5\,\mu  ,\]
for some~$C_5>0$.


Now we claim that
the sets $\mathcal G_{E}$ and $ \mathcal G_{\mathcal CE}$ are ordered
with respect to the vertical direction,
namely
there exist $t_*, t^* \in[-1,1]$, such that
\begin{equation}\label{ORD3}
{\rm ess\,sup}  \,\mathcal G_E = t_* \,\le\, t^*  ={\rm ess\,inf}\,  \mathcal G_{\mathcal CE}.
\end{equation}
For this, let~$t_1<t_2\in(-1,1)\setminus{\mathcal{B}}$. 
We show that
\begin{equation}\label{JH:12}
{\mbox{if $t_1\in \mathcal G_{\mathcal CE}$ then 
$t_2\in \mathcal G_{\mathcal CE}$.}}\end{equation}
We argue by contradiction, assuming that~$
t_2\not\in \mathcal G_{\mathcal CE}$. Then, by~\eqref{CO:11},
we obtain that~$t_2\in
 \mathcal G_{E}$. Consequently, by~\eqref{GE}, we have that $t_2^2 <1-r^2$ and
$$  H^{n-1}\bigl( \{x_n =t_2\}\cap B_1 \cap {\mathcal{C}}E \bigr)\le \mu.$$
We can write this as~$\chi_E(x',t_2)=1$ for any~$x'$ in the ball
$B_{\sqrt{1-t_2^2}}^{(n-1)}$ outside a set
of~$(n-1)$-dimensional measure less than~$\mu$
(so, in particular,
for any~$x'$ in the smaller ball
$B_r^{(n-1)}$ outside a set
of~$(n-1)$-dimensional measure less than~$\mu$).

Also, the condition~$t_1\in \mathcal G_{\mathcal CE}$
and~\eqref{GCE} give that $t_1^2 < 1-r^2$ and
$$ H^{n-1}\bigl( \{x_n =t_1\}\cap B_1 \cap  E \bigr)\le \mu.$$
We can write this as~$\chi_E(x',t_1)=0$ for any~$x'\in 
B_{\sqrt{1-t_1^2}}^{(n-1)}$ outside a set
of~$(n-1)$-dimensional measure less than~$\mu$
(so, in particular,
for any~$x'
\in B_r^{(n-1)}$ outside a set
of~$(n-1)$-dimensional measure less than~$\mu$).

By~\eqref{LA-02}, we also know
that $\chi_E(x',t)$ is nonincreasing in $t$ outside~${\mathcal{B}}_n$,
which is another set
of~$(n-1)$-dimensional measure less than~$\mu$.

This means that, for $x'\in B_r^{(n-1)}$ outside a
set
of~$(n-1)$-dimensional measure less than~$3\mu$, we have that
\begin{equation}\label{56n}
1=\chi_E (x',t_2)\le \chi_E(x',t_1)=0.\end{equation}
We stress the fact that this set to which $x'$ belongs
is nonvoid, since $H^{n-1}B_r^{(n-1)})$ is strictly bigger than $3\mu$,
thanks to \eqref{proportion}.
Therefore, the inequality in \eqref{56n} provides a contradiction.
This proves~\eqref{JH:12}. Similarly, one proves that
\begin{equation}\label{JH:13}
{\mbox{if $t_2\in \mathcal G_{E}$ then 
$t_1\in \mathcal G_{E}$.}}\end{equation}
By putting together~\eqref{JH:12} and~\eqref{JH:13},
one obtains~\eqref{ORD3}.

Then it readily follows that {\rm ({\bf F1})} is satisfied with $\varepsilon= C(n) \mu$.

{\em Step 3} We show that {\rm ({\bf F3})} with $\varepsilon= \mu$ in any dimension $n\ge2$.

Recall that we denote $F^\varepsilon = \{ (x', x_n/\varepsilon)\ :\ (x',x_n)\in F \}$.
Using Proposition \ref{prop-fubini-typeV} we estimate
\[
\begin{split}
{\rm Per}_{B_1^\varepsilon} (E^\varepsilon) &=
\sup\biggl\{ -\int_{B_1^\varepsilon}  \chi_{E^\varepsilon} \,{\rm div }\,\phi  \,dx \ :\ \
\phi \in C^1_c(B_1^\varepsilon;\R^n),\ | \phi|\le 1 \biggr\}
\\
&\le  \sum_{i=1}^{n-1}  \sup\biggl\{ -\int_{B_1^\varepsilon}  \chi_{E^\varepsilon}  \partial_i \psi \,dx \ : \psi\in C^1_c(B_1^\varepsilon), \ |\psi|\le 1 \biggr\} \ +
\\
&\hspace{35mm}+\sup\biggl\{ -\int_{B_1^\varepsilon}  \chi_{E^\varepsilon} \partial_n \psi \,dx \ : \psi\in C^1_c(B_1), \ |\psi|\le 1\biggr\}
\end{split}
\]
Then, using the change of variables $y'=x'$ and $y_n=\varepsilon x_n$, we have
\[
\begin{split}
{\rm Per}_{B_1^\varepsilon} (E^\varepsilon)
&\le \sum_{i=1}^{n-1}  \sup\biggl\{ -\int_{B_1}  \chi_E \, \partial_i \bar \psi \,\frac{dx}{\varepsilon}   \ : \bar \psi\in C^1_c(B_1), \ |\bar \psi|\le 1 \biggr\} \ +
\\
&\hspace{35mm}+\sup\biggl\{ -\int_{B_1}  \chi_{E} \,\varepsilon\partial_n \bar \psi \,\frac{dx}{\varepsilon}  \ : \bar \psi\in C^1_c(B_1), \ |\bar \psi|\le 1\biggr\}
\\
&\le \frac{2}{\varepsilon} \sum_{i=1}^{n-1} \max\bigl\{ \Phi_+(\boldsymbol e_i), \Phi_-(\boldsymbol e_i)\bigr\}  \ +
\\
&\hspace{35mm}
+\min\bigl\{ \Phi_+(\boldsymbol e_n), \Phi_-(\boldsymbol e_n)\bigr\}  +2\bigl| \Phi_+(\boldsymbol e_n) - \Phi_-(\boldsymbol e_n)\bigr|
\end{split}
\]

Now we use that 
$$ \Phi_+(\boldsymbol e_n) - \Phi_-(\boldsymbol e_n) = \int_{B_1}\partial_{\boldsymbol e_n} u =\int_{\partial B_1} u(x) \nu_n(x)dH^{n-1}(x).$$
Hence,
$$ |\Phi_+(\boldsymbol e_n) - \Phi_-(\boldsymbol e_n)|\leq \int_{\partial B_1} | \nu_n(x)| dH^{n-1}(x)=2|B_1^{(n-1)}|.$$

Thus, taking $\varepsilon = \mu$ and using 
\eqref{max-phi} 
and \eqref{initial-point2},
we obtain
\[ {\rm Per}_{B_1^\varepsilon} (E^\varepsilon) \le 2(n-1) + \varepsilon + 4\,|B_1^{(n-1)}| \le c(n)\]
and {\rm ({\bf F3})} follows.
\end{proof}

We now give the

\begin{proof}[Proof of Theorem \ref{flatness}]
For the proof we need to combine Lemmas \ref{lem2A}, \ref{lemEFtdelta}, \ref{prop:intermediate} and \ref{geometric-info}.

More precisely, using Lemma \ref{lem2A}, point (a), and Lemma \ref{lemEFtdelta}, we find that for any $\varepsilon>0$ there exists $t_0$ such that for $t\in(0,t_0)$
\[ \min\bigl\{ L_K(E_{R,t}\setminus E, E\setminus E_{R,t})\,,\,  L_K(E_{R,-t}\setminus E, E\setminus E_{R,-t})\bigr\}\le  (\eta/4 +\varepsilon)t^2,\]
where $E_{R,t}$ is defined as in \eqref{def-E+} and 
$$\eta = \frac{32}{R^2} P_{K^*,B_R}(E).$$
This implies that for all $\boldsymbol v$ there is some sequence $t_k\to 0$, $t_k\in(-1,1)$ such that
\[ \lim_{k\to \infty} t_k^{-2} L_K(E_{R,t_k}\setminus E, E\setminus E_{R,t_k}) \le  \eta/4.\]

Now, by definition of $E_{R,t}$ we have
$$E_{R,t}\cap B_1=(E+t\boldsymbol v)\cap B_1.$$
and thus using the assumption \eqref{Kboundedbelow} ---i.e.  $K\ge1$ in $B_2$--- we obtain 
\[ \lim_{k\to \infty} t_k^{-2} |(E +t_k\boldsymbol v)\setminus E|\cdot| E\setminus (E +t_k\boldsymbol v)| \le  \eta/4.\]

Therefore, applying Lemma \ref{prop:intermediate} we obtain 
\[
\min\{ \Phi_+(\boldsymbol v),\Phi_-(\boldsymbol v)\} \le \sqrt \eta /2,
\]
where $\Phi_\pm (\boldsymbol v) = (-\partial_{\boldsymbol v}u)_\pm(B_1)$.

Then, applying Lemma \ref{geometric-info} we obtain that $E$ satisfies ({\bf F1}), ({\bf F2}), and ({\bf F3}) with
\[ \varepsilon = C(n)\sqrt \eta = C(n)\sqrt{\frac{P_{K^*,B_R}(E)}{R^2} }.\]

The same inequality for $\varepsilon =  \frac{C(n)}{ \sqrt{\log R}}\sup_{\rho\in[1,R]} \sqrt{\frac{P_{K^*,B_\rho}(E)}{\rho^2} }$ is proved likewise using 
$\tilde E_{R,t}$ instead of $E_{R,t}$ and part (b) of Lemma \ref{lem2A} instead of part (a).
\end{proof}
Theorem \ref{aniso} and Corollaries \ref{thm1L1}, \ref{thm1dim3}, \ref{compt-supp} all follow by Theorem \ref{flatness} and estimate for the quantity $P_{K^*,B_R}(E)$.  

\begin{proof}[Proof of Theorem \ref{aniso}]
Multiplying the kernel $K\in \mathcal L_2$ by a positive constant, we may assume that $\lambda \geq 2^{n+s}$ and hence $K$ satisfies \eqref{Knonnegative}--\eqref{est-second-deriv} with $K^* = C_1 K$. Applying Corollary \ref{thmstable1}, we deduce that
\begin{equation}\label{est-stable}
P_{K^*,B_R}(E)=C_1P_{K,B_R}(E)\leq C R^{n-s}.\end{equation}
Thus, Theorem \ref{aniso} follows by Theorem \ref{flatness} and estimate \eqref{est-stable} above.
\end{proof}
\begin{proof}[Proof of Corollary \ref{thm1L1}]
The proof follows by Theorem \ref{flatness}, after observing that if $K^*\in L^1(\R^n)$ and $E$ is a minimizer, then
$$P_{K_*,B_\rho}(E)\leq |B_\rho|\int_{\R^2}K^*=\rho^2|B_1| \int_{\R^2}K^*.$$
\end{proof}

\begin{proof}[Proof of Corollary \ref{thm1dim3}]
The proof follows by Theorem \ref{flatness} and by Proposition \ref{est-P_K^*}.
\end{proof}

\begin{proof}[Proof of Corollary \ref{compt-supp}]
The proof follows by Corollary \ref{thm1dim3} using that for compactly supported kernels $K$, we have
\begin{equation*}P_K(B_R)\sim  R^{n-1}.\qedhere\end{equation*}
\end{proof}

\section{Energy estimates with perturbed kernels}\label{perturbed-K}

\begin{lem}\label{estimate_by_below_perimeter}
Let $R_0\ge1$. Assume that $K\ge 1$ in $B_1$. Let $Q = (-3R_0/2,3R_0/2)^n$ and $E\subset \R^n$ is measurable.
Then,
\[L_K( E\cap Q, \mathcal CE \cap Q) \ge c(n,R_0)\min\{| E\cap Q|, |\mathcal CE \cap Q|\}.\]
\end{lem}

\begin{proof}
Since the statement is invariant when we swap $E$ and $\mathcal CE$ we may assume $|E\cap Q|\le |Q|/2\le |\mathcal CE \cap Q|$.

Split $Q$ into a regular grid composed by $k^n$ open cubes of side $r= 3R_0/k$ with $ r\in(n^{-1/2}/8 , n^{-1/2}/4]$.
We call these small cubes $Q_i$, $i\in I$.
Let $\tilde I= \{ i\,:\, |Q_i\cap E| > \frac 12|Q_i|\}$.

We have $\tilde I\neq I$ since $|E\cap Q|\le |Q|/2$. There are now two cases $\tilde I$ nonempty or empty.

On the one hand, if $\tilde I$ is nonempty then there are $i_1\in \tilde I$ and $i_2\in I\setminus \tilde I$ such that $Q_{i_1}$ and $Q_{i_2}$ are adjacent cubes.
Then, since $r\le n^{-1/2}/4$ we have ${\rm diam}\,(Q_i)\le 1/2$ for all $i$ and thus
${\rm diam}\,(Q_{i_1}\cup Q_{i_2})\le 1$. Since $K\ge 1$ in $B_1$ we then have
\[\begin{split}
 L_K(E\cap Q, \mathcal CE \cap Q) &\ge L_K(E\cap Q_{i_1}, \mathcal CE \cap Q_{i_2} )
\ge \bigl|E\cap Q_{i_1}\bigr| \,\cdot\, \bigl|\mathcal C E\cap Q_{i_2}\bigr| \\
&\ge \frac 1 2|Q_{i_1}| \,\cdot\, \frac 1 2| Q_{i_2}| \ge c(n).
\end{split}\]
On this case the estimate of the lemma follows since $|E \cap Q|\le (3R_0)^n$.

On the other hand, if $\tilde I$ is empty then $|Q_i\cap E| \le \frac 12|Q_i|$ for all $i$ and
\[\begin{split}
 L_K(E \cap Q, \mathcal CE \cap Q) &\ge \sum_i L_K(E\cap Q_{i}, \mathcal CE \cap Q_{i} )
\ge  \sum_i |E\cap Q_{i}| \,\cdot\, \frac 1 2| Q_{i}| \\
&\ge c(n) |E\cap Q|,
\end{split}\]
as desired.
\end{proof}

\begin{lem}\label{energyestimatetilde}
Let $R_0\ge1$. Let $K$ be some kernel satisfying $K\ge 1$ in $B_1$. Let $E\subset\R^n$ measurable and $R \in 3R_0\mathbb N$ and $Q_R = (-R/2,R/2)^n$. Denote $ K_0(z) = \chi_{\{|z|\le R_0\}} (z)$.
Then,
\[L_{K_0}(E\cap Q_R, \mathcal C E\cap Q_R) \le C(n,R_0) L_K(E\cap Q_R, \mathcal C E\cap Q_R).\]
\end{lem}

\begin{proof}
Let us cover the full measure of $Q_R$ by cubes belonging to the grid of disjoint open cubes of size $R_0$ given by $\{Q_i\} \subset R_0\bigl(\mathbb Z^n + (-1/2, 1/2)^n\bigr)$. Let us consider also the covering of $Q_R$ by cubes in the overlapping grid of side $3R_0$ given by $\{\bar Q_i\} \subset R_0\bigl(\mathbb Z^n + 3(-1/2, 1/2)^n\bigr)$.
Note that (up to sets of measure zero) each point of $Q_R$ belongs to exactly one cube in $\{Q_i\}$ and $3^n$ cubes in $\{\bar Q_i\}$.

Notice that for every pair of points $x,\bar x \in Q_R$ such that $|x-\bar x|\le R_0$ there is some large cube $\bar Q_i$ containing at the same time both points. Indeed, $x$ will belong to some small cube $Q_i$ but then if $\bar Q_i$ is the large cube with the same center it will also be $y\in \bar Q_i$. Hence,
\[ \{(x,\bar x) \in Q_R\times Q_R \,:\, |x-\bar x|\le R_0\}\subset \bigcup _i \bar Q_i\times \bar Q_i. \]
This implies that
\[\begin{split}
L_{ K_0}(E\cap Q_R, \mathcal C E\cap Q_R) &= \iint_{(E\cap Q_R)\times(\mathcal C E\cap Q_R)}  \chi_{\{|\bar x -x|\le R_0\}} \,dx \,d\bar x \\
& =   \iint_{ \bigcup _i (E\cap \bar Q_i)\times(\mathcal C E\cap \bar Q_i)} \chi_{\{|\bar x -x|\le R_0\}} \,dx \,d\bar x\\
& \le \sum _i \iint_{ (E\cap \bar Q_i)\times(\mathcal C E\cap \bar Q_i)} \chi_{\{|\bar x -x|\le R_0\}} \,dx \,d\bar x\\
&= \sum_i L_{ K_0}(E\cap \bar Q_i, \mathcal C E\cap \bar Q_i).\\
\end{split}
\]

Now, using Lemma \ref{estimate_by_below_perimeter} we obtain that,  for all $i$,
\[\begin{split}
L_{ K_0}(E\cap \bar Q_i, \mathcal C E\cap \bar Q_i) &\le
|E\cap \bar Q_i|\,\cdot\, |\mathcal C E\cap \bar Q_i| \\
&\le (3R_0)^n \min\bigl\{|E\cap \bar Q_i|\,\cdot\, |\mathcal C E\cap \bar Q_i|\bigr\}
\\
&\le C(n,R_0) L_{K}(E\cap \bar Q_i, \mathcal C E\cap \bar Q_i).
\end{split}
\]

But then, using that each point of $B_R$ belongs to at most $3^n$ cubes $\bar Q_i$ we can estimate
\[\begin{split}
L_{ K_0}(E\cap Q_R, \mathcal C E\cap Q_R) &\le \sum_i L_{ K_0}(E\cap \bar Q_i, \mathcal C E\cap \bar Q_i)\\
& \le \sum_i C(n,R_0) L_{K}(E\cap \bar Q_i, \mathcal C E\cap \bar Q_i) \\
& \le C(n,R_0) 3^n L_K(E\cap Q_R, \mathcal C E\cap Q_R),
\end{split}
\]
as stated in the Lemma.
\end{proof}

We finally give the

\begin{proof}[Proof of Proposition \ref{est-P_K^*}]
Note first that all $R\ge 1$ we have $P_K(B_R) \ge c(n) R^{n-1}$ since $K\ge 1$ in $B_2$ by \eqref{Kboundedbelow}.
On the other hand it is clear that by definition $P_K(B_R)$ is monotone in $R$.

Thus, if $k$ is the smallest integer such that $3R_0k/2\ge R$, denoting $\bar R = 3R_0k/2$,  $E_R = E\cap B_R$. Denote $K_0(z) = \chi_{\{|z|\le R_0\}} (z)$. Using Lemma \ref{energyestimatetilde} we obtain
\[
\begin{split}
P_{K_0}(E,B_R)  & \le L_{ K_0}(E\cap B_{R}, \mathcal C E\cap B_R) + L_{ K_0}(B_{R},  C B_{R})\\
& \le L_{ K_0}(E_R\cap Q_{\bar R}, \mathcal C E_R \cap Q_{\bar R})  + \int_{B_R} \int_{\mathcal C B_{R}} \chi_{\{|\bar x-x|\le R_0\}}\,d\bar x\,dx \\
&\le C(n,R_0) L_{K}(E_R\cap Q_{\bar R}, \mathcal C E_R \cap Q_{\bar R}) + C(n,R_0)  R^{n-1} \\
&\le C(n,R_0)\bigl( L_{K}(E\cap B_R, \mathcal C E\cap B_R)+ L_{K}(B_{R},  C B_{R})+  R^{n-1}\bigr)\\
&\le C(n,R_0)\bigl( P_{K,B_R}(E)+ P_K(B_R)+  R^{n-1}\bigr),\\
&\le C(n,R_0)  P_K(B_R)
\end{split}
\]
where we have used that $E$ is a minimizer $P_{K,B_R}$ and hence $P_{K,B_R}(E)\le P_K(B_R)$. Then the proposition follows.
\end{proof}

\section{Existence and compactness of minimizers}

To prove Theorem \ref{existence} we need some preliminary results. First we prove existence of minimizers among ``nice'' sets (more precisely among sets with finite $1/2$-perimeter); this is done in Proposition \ref{exist-restricted}, where a crucial ingredient in the proof is given by the uniform $BV$-bound established in Theorem \ref{thm0} which provides the necessary compactness in $L^1$. Second, we establish a density result (see Proposition \ref{smooth-dense}) which allows to approximate any set of finite $K$-perimeter, with sets that has also finite $1/2$-perimeter; the proof of this density result uses a generalized coarea formula that we establish in Lemma \ref{coarea}.

We start with a simple remark which will be useful in the sequel.

\begin{prop}[\textbf{Lower semicontinuity of $K$-perimeter}]\label{lsc}
Let $\chi_{E_k}\rightarrow \chi_E$ in $L^1_{\rm loc}(\R^n)$, then
\[ \liminf_{k\rightarrow \infty} P_{K,\Omega}(E_k)\ge P_{K,\Omega}(E).\]
\end{prop}
\begin{proof}
The result follows, exactly as in Proposition 3.1 in \cite{CRS}, by Fatou Lemma. Indeed, recall that
$$L_K(A,B)=\int_{\R^n}\int_{\R^n} \chi_A(x)\chi_B(\bar x) K(x-\bar x) dx d\bar x.
$$
If $\chi_{A_k} \rightarrow \chi_A$, $\chi_{B_k}\rightarrow \chi_B$ in $L^1_{loc}(\R^n)$, then for each sequence there exists a subsequence $k_j$, such that for a.e. $(x,\bar x)$
$$\chi_{A_{k_j}}\chi_{B_{k_j}}\rightarrow \chi_A \chi_B.$$
Therefore, by Fatou Lemma, we have
$$\liminf_{j\rightarrow \infty} L_K(A_{k_j},B_{k_j})\geq L_K(A,B).$$
\end{proof}

In the next lemma we establish a generalized coarea formula for the $K$-perimeter. The analogue result for the fractional $s$-perimeter is contained in \cite{Vis}. For the sake of completeness, we reproduce here the simple proof, which does not dependent on the choice of the kernel. For a measurable function $u$, we set:
$$\mathcal F_{K,\Omega}(u):=\frac{1}{2}\int_{\Omega}\int_{\Omega}|u(x)-u(\bar x)| K(x-\bar x)dx d\bar x + \int_{\Omega}\int_{\mathcal C \Omega}|u(x)-u(\bar x)| K(x-\bar x)dx d\bar x.$$
\begin{lem}[\textbf{Coarea formula}]\label{coarea}
Let $u:\Omega\rightarrow [0,1]$ be a measurable function. Then, we have
$$\mathcal F_{K,\Omega}(u)=\int_0^1 P_{K,\Omega}(E_t) dt,$$
where $E_t=\{u>t\}$.
\end{lem}
\begin{proof}
We start by observing that the function $t \mapsto \chi_{E_t}(x)-\chi_{E_t}(\bar x)$ takes values in $\{-1,0,1\}$ and it is different from zero in the interval having $u(x)$ and $u(\bar x)$ as extreme points, therefore
$$|u(x)-u(\bar x)|=\int_0^1 |\chi_{E_t}(x)-\chi_{E_t}(\bar x)| dt.$$

Hence, by Fubini Theorem, we deduce
\begin{eqnarray*}
\mathcal F_{K,\Omega}(u)&=&\int_0^1 \left[ \frac{1}{2}\int_{\Omega}\int_{\Omega} |\chi_{E_t}(x)-\chi_{E_t}(\bar x)| K(x-\bar x)dx d\bar x\right] dt \\
&&\hspace{1em} +\int_0^1 \left[\int_{\Omega}\int_{\mathcal C\Omega} |\chi_{E_t}(x)-\chi_{E_t}(\bar x)| K(x-\bar x)dx d\bar x\right] dt \\
&=& \int_0^1 \left[\int_{E_t \cap \Omega}\int_{\mathcal C E_t\cap \Omega}K(x-\bar x) dx d\bar x + \int_{E_t \cap \Omega}\int_{\mathcal C E_t\cap \mathcal C\Omega}K(x-\bar x) dx d\bar x\right.\\
&&\hspace{1em}+\left.\int_{\mathcal C E_t \cap \Omega}\int_{E_t\cap \mathcal C \Omega}K(x-\bar x) dx d\bar x\right]dt\\
&=& \int_0^1  [L_K(E_t\cap \Omega,\mathcal C E_t\cap \Omega)+ L_K(E_t\cap \Omega,\mathcal C E_t\setminus \Omega) + L_K(E_t\setminus \Omega,\mathcal C E_t \cap \Omega)]dt\\
&=& \int_0^1 P_{K,\Omega}(E_t) dt,
\end{eqnarray*}
as desired.
\end{proof}

In the following lemma we establish a density result for smooth functions in the space of functions with finite $\mathcal F_{K,\Omega}$. For the sake of completeness we reproduce here the simple proof, which follows the one in \cite{FSV}, Lemma 11, for the case of the all space.

\begin{lem}\label{dense-u}
Let $\Omega$ be a bounded Lipschitz domain and $u$ be a function defined on $\R^n$ with $u\in L^1(\Omega)$ and $\mathcal F_{K,\Omega}(u)< \infty$.
Then, for any fixed sufficiently small $\delta>0$, there exists a family $(u_{\epsilon})$ of smooth functions such that:
\begin{itemize}
\item[i)] $\|u-u_{\epsilon}\|_{L^1(\Omega^\delta)}\rightarrow 0$ as $\epsilon\rightarrow 0$,
\item[ii)] $\mathcal F_{K,\Omega^\delta}(u-u_{\epsilon}) \rightarrow 0$ as $\epsilon\rightarrow 0$.
\end{itemize}
\end{lem}

\begin{proof}
For any $0<\epsilon <\delta$, we consider the convolution kernel
$$\eta_\varepsilon(x):=\epsilon^{-n}\eta\left(\frac{x}{\epsilon}\right),$$
where $\eta \in C^\infty_0(B_1)$, $\eta \geq 0$, $\int_{\R^n} \eta =1$, and we set
$$u_\epsilon (x):=(u * \eta_\epsilon)(x).$$
Since $u \in L^1(\Omega)$ we immediately have $||u-u_\epsilon||_{L^1(\Omega^\delta)}\rightarrow 0$ as $\epsilon \rightarrow 0$. It remains to prove ii).

Using the definition of $u_\epsilon$ and the triangle inequality, we have that
\begin{eqnarray*}
2\mathcal F_{K,\Omega^\delta}(u_\epsilon -u)&=&\iint_{\R^{2n}\setminus (\mathcal C \Omega^\delta\times \mathcal C\Omega^\delta)}|u_\epsilon(x)-u(x)+u(\bar x)-u_\epsilon(\bar x)| K(x-\bar x) dx d\bar x\\
&&\hspace{-7em}= \iint_{\R^{2n}\setminus (\mathcal C \Omega^\delta\times \mathcal C\Omega^\delta)} K(x-\bar x) \left |\int_{B_1} (u(x-\epsilon  z)-u(\bar x-\epsilon  z)-u(x)+u(\bar x) )\eta( z) d z \right| dx dy\\
&&\hspace{-7em}\leq \int_{B_1}\iint_{\R^{2n}\setminus (\mathcal C \Omega^\delta\times \mathcal C\Omega^\delta)} K(x-\bar x)  |u(x-\epsilon  z)-u(\bar x-\epsilon z)-u(x)+u(\bar x)| \eta( z)  dx dy dz.
\end{eqnarray*}
Now, by the continuity of translations in $L^1(\R^{2n}\setminus \mathcal (C\Omega^\delta\times \mathcal C\Omega^\delta))$ applied to the function
$$v(x,\bar x)=(u(x)-u(\bar x))K(x-\bar x),$$
(which is in $L^1(\R^{2n}\setminus \mathcal (C\Omega^\delta\times \mathcal C\Omega^\delta))$, since $\mathcal F_{K,\Omega}(u)<\infty$), we deduce that for every fixed $ z \in B_1$,
$$
\iint_{\R^{2n}\setminus (\mathcal C \Omega^\delta\times \mathcal C\Omega^\delta)} K(x-\bar x)  |u(x-\epsilon  z)-u(\bar x-\epsilon  z)-u(x)+u(\bar x)| dx dy\rightarrow 0,\;\;\;\mbox{as}\;\;\epsilon \rightarrow 0.
$$
Moreover, for a.e. $ z \in B_1$, we have
\begin{equation*}
\begin{split}
&\eta( z) \iint_{\R^{2n}\setminus (\mathcal C \Omega^\delta\times \mathcal C\Omega^\delta)} K(x-\bar x)  |u(x-\epsilon  z)-u(\bar x-\epsilon  z)-u(x)+u(\bar x)| dx dy\\
&\leq 2\max\eta \iint_{\R^{2n}\setminus (\mathcal C \Omega^\delta\times \mathcal C\Omega^\delta)} K(x-\bar x)  |u(x)-u(\bar x)| dx dy<\infty.
\end{split}
\end{equation*}
Hence, the conclusion follows by the dominated convergence Theorem.
\end{proof}

The following density result will be useful in the proof of existence of minimizers. The proof follows the one for the classical approximation result for sets of finite perimeter by smooth sets, and uses the generalized coarea formula of Lemma \ref{coarea}.
\begin{prop}[\textbf{Density of sets with finite $1/2$-perimeter}]\label{smooth-dense}
Let $\Omega$ be a bounded Lipschitz domain. Let $F$ be a set with finite $K$-perimeter in $\Omega$. Then, there exists a sequence $(F_j)$ of open sets satisfying the following properties:
\begin{enumerate}
\item $P_{1/2,\Omega}(F_j)< \infty$,
\item $F_j\setminus \Omega=F\setminus \Omega$,
\item $\lim_{j\rightarrow \infty}|F_j \triangle F|=0$,
\item $\lim_{j\rightarrow \infty}P_{K,\Omega}(F_j)=P_{K,\Omega}(F)$.
\end{enumerate}
\end{prop}
To prove Proposition \ref{smooth-dense}, we need the following preliminary result.

Let $\Omega$ be a Lipschitz domain and let $d(x,\partial \Omega)$ denote the distance of the point $x$ from the boundary $\partial \Omega$. We define
\begin{equation}\label{Omegat}
\Omega^t:=\{x \in \Omega:\;d(x,\partial \Omega)>t\} \quad \mbox{and}\quad
\end{equation}
Note that for a sufficiently small $\delta_0>0$, where $t\in (0,\delta_0)$ all the domains $\Omega^t$ are Lipschitz with uniform constants depending only on $\Omega$.
We will need the following lemma.
\begin{lem}\label{lemOmegat}
Let $\Omega\subset \R^n$ be a Lipschitz domain and suppose that  $K\geq 0$ satisfies assumption \eqref{Kintegrability}. There exists $\delta_0 >0$ depending only on $\Omega$ such that for any $t\in(0,\delta_0)$ we have  
\begin{equation}\label{intK}
L_K(\Omega\setminus \Omega^t, \Omega^t)\leq C\int_{\R^n}\min\{t,|z|\}K(z)dz,
\end{equation}
and
\begin{equation}\label{intK2}
L_K(\Omega\setminus \Omega^t,\mathcal C\Omega)\leq C\int_{\R^n}\min\{t,|z|\}K(z)dz,
\end{equation}
where the constants $C$ and $\delta$ depend only on $\Omega$.
\end{lem}
\begin{proof}
Performing the change of variables $z=\bar x-x$ and using Fubini Theorem, we have
\begin{equation*}
\begin{split}
L_K(\Omega\setminus \Omega^t,\Omega^t)&= \int_{\Omega^t}dx \int_{\Omega \setminus \Omega^t}d\bar x K(\bar x-x)=\int_{\R^n}dz K(z) \int_{\Omega^t\cap \big((\Omega\setminus \Omega^t)-z\big)} dx \\
&\leq C\int_{\R^n}\min\{t,|z|\}K(z)dz,
\end{split}
\end{equation*}
since for a Lipschitz set $\Omega$, we have 
$$
|  \Omega^t\cap \big((\Omega\setminus \Omega^t)-z\big)|\leq \min\big\{|\Omega \setminus \Omega^t|\ ,\ |\Omega^t\setminus (\Omega^t-z)| \big\}\leq \min \{t , |z|\}.
$$
The proof of \eqref{intK2} follows likewise.
\end{proof}
\begin{proof}[Proof of Proposition \ref{smooth-dense}]
As it will become clear in the proof, actually we prove more than property (1): we will show that for any $j$, on the one hand $\partial F_j$ is smooth  in $\Omega^{\frac{1}{j}}$ (and up to the boundary of $\Omega^{\frac{1}{j}}$), and on the other hand $F_j\cap \Omega\setminus\Omega^{\frac{1}{j}}$. Recall that $\Omega^{\frac{1}{j}}$ was defined in \eqref{Omegat}. Since $\Omega$ is Lipschitz, these two properties imply that $F_j$ satisfies (1), for $j$ large enough.

For a fixed sufficiently small $\delta$, we consider $\Omega^\delta$ ---as in \eqref{Omegat}.
Let $\epsilon_k \in (0,\delta)$ be a sequence such that $\epsilon_k\downarrow 0$, and let $u_k$ be the mollified functions
$$u_k:=\chi_F * \eta_{\epsilon_k}.$$
By Lemma \ref{dense-u} we know that
\begin{equation}\label{volume}
\|u_k -\chi_F\|_{L^1(\Omega^\delta)} \rightarrow 0,\quad \mbox{as}\;\;k \rightarrow \infty,
\end{equation}
and
\begin{equation*}
\lim_{k\rightarrow \infty}\mathcal F_{K,\Omega^\delta}(u_k)=\mathcal F_{K,\Omega^\delta}(\chi_F)=P_{K,\Omega^\delta}(F).
\end{equation*}
We define now the sets
$$ F_t^k:=\{u_k >t\}.$$
By the coarea formula of Lemma \ref{coarea}, we have that
\begin{equation*}
\begin{split}
P_{K,\Omega^\delta}(F)&=\lim_{k\rightarrow \infty}\mathcal F_{K,\Omega^\delta}(u_k)\\
&\geq\int_0^1 \liminf_{k \rightarrow \infty}P_{K,\Omega^\delta}(F_t^k)dt.
\end{split}
\end{equation*}
Sard's Theorem implies that for $\mathcal L^1$-a.e. $t\in (0,1)$, all the sets $F_t^k$ have smooth boundary, therefore we can choose $t$ with this property and such that
$$L:=  \liminf_{k \rightarrow \infty}P_{K,\Omega^\delta}(F_t^k)\leq P_{K,\Omega^\delta}(F).$$
Let now $(F_h)=(F_t^{k(h)})$ be a subsequence with finite $K$-Perimeter in $\Omega^\delta$ converging to $L$. By Chebyshev inequality and \eqref{volume} we deduce that
\begin{equation}\label{volume1}
|(F_h\triangle F)\cap \Omega^\delta|\rightarrow 0, \quad \mbox{as}\;\;h \rightarrow \infty.
\end{equation}
Moreover by the lower semicontinuity of the $K$-Perimeter, we deduce that
\begin{equation}\label{per1}
\lim_{h\rightarrow \infty} P_{K,\Omega^\delta}(F_h)=P_{K,\Omega^\delta}(F).
\end{equation}
We define now the sequence of sets
\begin{equation}\label{def-F_h}
F_h^\delta:=(F_h\cap \Omega^\delta)\cup A^\delta \cup (F\setminus \Omega).
\end{equation}
We start by observing that, by definition, $F_h^\delta$ satisfies
\begin{equation}\label{1e2}
F_h^\delta\setminus \Omega=F\setminus \Omega,\quad \mbox{and}\quad F_h^\delta\;\;\mbox{is smooth in}\;\;\overline{\Omega^\delta}.
\end{equation}

Moreover, using Lemma \ref{dense-u} and that $|A^\delta|=C\delta$, we see that
\begin{equation}\label{3}
\lim_{h\rightarrow \infty}|F_h^\delta\triangle F|=C\delta.
\end{equation}
Here and in the sequel $C$ denotes possibly different positive constant (uniform in $h$ and $\delta$).
We estimate now how much the $K$-perimeters of $F$ and $F_h^\delta$ differs.
By the triangle inequality, we have that
\begin{equation}\label{chain}
\begin{split}
|P_{K,\Omega}(F_h^\delta)-P_{K,\Omega}(F)|&\leq |P_{K,\Omega}(F_h^\delta)-P_{K,\Omega^\delta}(F_h^\delta)|+|P_{K,\Omega^\delta}(F_h^\delta)-P_{K,\Omega^\delta}(F_h)|\ +\\
&\hspace{3em}+|P_{K,\Omega^\delta}(F_h)-P_{K,\Omega^\delta}(F)|+ |P_{K,\Omega^\delta}(F)-P_{K,\Omega}(F)|\\
&= I_1+I_2+I_3+I_4.
\end{split}
\end{equation}
We readiy show that, for $i=1,2,4$,
\begin{equation}\label{per2}
I_i\leq L_K(\Omega\setminus \Omega^\delta,\Omega^\delta) + L_K(\Omega\setminus \Omega^\delta, \mathcal C \Omega^\delta).
\end{equation}
Using Lemma \ref{lemOmegat} we deduce that
$$I_i\leq C\int_{\R^n}K(z)\min\{\delta,|z|\}dz,$$
where $C$ depends only on $\Omega$.
Finally, by point (4) in Proposition \ref{smooth-dense},  we have that for any fixed $\delta$,
\begin{equation}\label{per3}
I_3\rightarrow 0\quad \mbox{as}\;\;h\rightarrow \infty.
\end{equation}

Let now $j$ be given. We choose $\delta=\delta(j)$ such that $I_i\leq 1/(4j)$, for $i=1,2,4$. Moreover, by \eqref{per3}, we can choose $h=h(j)$ such that $I_3\leq 1/(4j)$.
Finally we set $F_j:=F_{h(j)}^{\delta(j)}$. With this choices, plugging \eqref{per2},\eqref{per3} in \eqref{chain} we deduce that
$$|P_{K,\Omega}(F_j)-P_{K,\Omega}(F)|\leq \frac{1}{j}.$$
In addition, by \eqref{1e2} and \eqref{3}, we have that $F^j$ has smooth boundary in $\Omega^{\frac{1}{j}}$ and is such that
$$F_j\setminus \Omega=F\setminus \Omega,\quad |F_j\triangle F|\leq \frac{1}{j}.$$

To conclude the proof, it remains therefore to show (1). This is an easy consequence of the fact that $F_j$ has smooth boundary in $\Omega^{\frac{1}{j}}$.
Indeed, given any set $\tilde F$ with smooth boundary in $\Omega^{\frac{1}{j}}$, and using again Lemma \ref{lemOmegat}, we have
\begin{equation*}
\begin{split}
P_{1/2,\Omega}(\tilde F)&=P_{1/2,\Omega^{\frac{1}{j}}}(\tilde F)+C\int_{\R^n}K(z)\min\left\{\frac{1}{j},|z|\right\}dz\\
&=\int_{\tilde F\cap\Omega^{\frac{1}{j}}}\int_{\mathcal C\tilde F\cap \Omega^{\frac{1}{j}}}\frac{1}{|x-\bar x|^{n+s}}dx d\bar x +
2\int_{\Omega^{\frac{1}{j}}}\int_{\mathcal C \Omega^{\frac{1}{j}}}\frac{1}{|x-\bar x|^{n+s}}dx d\bar x +C <\infty,
\end{split}
\end{equation*}
as desired.
\end{proof}

\begin{prop}[\textbf{Existence of minimizers among ``nice'' sets}]\label{exist-restricted}
Let $\Omega$ be a bounded Lipschitz domain, and $E_0\subset \cC\Omega$ a given set. Then, there exists a set $E$, with $E\cap \cC\Omega=E_0$ that is a minimizer for $P_{K,\Omega}$ among all sets $F$ with $P_{1/2,\Omega}(F)\leq C$.
\end{prop}

\begin{proof}
Let $\epsilon>0$. We introduce the following regularized kernel:
$$K_\epsilon(z):=K(z)+\frac{\epsilon}{|z|^{n+\frac{1}{2}}}.$$

For any $\epsilon$ fixed, the associated perimeter $P_{K_\epsilon,\Omega}$ admits a minimizer $E_\epsilon$ with $E_\epsilon\cap \cC\Omega=E_0$. This follows as in the proof of Theorem 3.2 in \cite{CRS} by $L^1$-compactness of $H^{\frac{1}{4}}$ and the lower semicontinuity of $P_{K_\epsilon,\Omega}$ (that follows by Proposition \ref{lsc} applied to $P_{K_\epsilon,\Omega}$ in place of $P_{K,\Omega}$). Indeed given $F_{\epsilon,k}$ a sequence of sets such that
$$P_{K_\epsilon, \Omega}(F_{\epsilon,k}) \underset{k  \rightarrow \infty}{\longrightarrow} \inf_{F\cap \cC\Omega=E_0} P_{K_\epsilon,\Omega}(F),$$
then the $H^{\frac{1}{4}}$-norm of the characteristic functions of $F_{\epsilon,k} \cap \Omega$ are bounded (by a constant depending on $\epsilon$), thus, by compactness, there exists a subsequence which converges to a set $E_\epsilon \cap \Omega$ in $L^1(\R^n)$, which is a minimizer of $P_{K_\epsilon,\Omega}$ by lower semicontinuity.

Now we observe that the new kernel $K_\epsilon$ satisfies all assumptions \eqref{Knonnegative}--\eqref{Kboundedbelow} and \eqref{est-second-deriv}, therefore, by Theorem \ref{thm0} and a standard covering argument, we have a uniform $BV$-bound (uniform in $\epsilon$!) for the characteristic functions of the minimizers $E_\epsilon$ in any subdomains $\Omega'$, with $\overline {\Omega '}\subset \Omega$. We set, as above, $\Omega^\delta=\{x\in \Omega:\:d(x,\partial \Omega)>\delta\}$. 

Using that $BV$ is compact in $L^1$ and the standard diagonal argument, we can extract a subsequence $\epsilon_j$ such that
$$\chi_{E_{\epsilon_j,}}\rightarrow \chi_E\quad \mbox{in}\;\;L^1(\Omega^\delta)\quad \mbox{for all}\;\;\delta>0.$$

It remains to prove that $E$ is a minimizer for $P_{K,\Omega}$. On one hand, by definition of $K_\epsilon$ and by the lower semicontinuity of $P_{K,\Omega}$, we have
\begin{equation}\label{liminf}
\liminf_{\epsilon \rightarrow 0} P_{K_\epsilon,\Omega^\delta}(E_\epsilon)\geq \liminf_{\epsilon \rightarrow 0} P_{K,\Omega^\delta}(E_\epsilon) \geq P_{K,\Omega^\delta}(E).
\end{equation}

On the other hand, by minimality of $E_\epsilon$, we have that
\begin{equation}\label{limsup}
P_{K_\epsilon,\Omega}(E_\epsilon)\leq P_{K_\epsilon, \Omega}(F),
\end{equation}
for any measurable set $F$ with $F\cap \cC\Omega=E_0$.

Hence, we deduce that
\begin{equation*}
\begin{split}
P_{K,\Omega^\delta}(E)&\leq \liminf_{\epsilon \rightarrow 0}P_{K_\epsilon,\Omega^\delta}(E_\epsilon)\nonumber\\
&\leq P_{K_\epsilon,\Omega}(F)\\
&=P_{K,\Omega}(F)+\epsilon P_{1/2,\Omega}(F_\delta)\nonumber
\end{split}
\end{equation*}

When a $P_{1/2,\Omega}(F)<\infty$, the conclusion then follows by sending first $\epsilon$ to zero and then $\delta$ to zero.
\end{proof}
We can now give the proof of our existence result.
\begin{proof}[Proof of Theorem \ref{existence}]
The theorem follows combining Propositions \ref{exist-restricted} and \ref{smooth-dense}.

\end{proof}

\begin{lem}[\textbf{Compactness}]\label{compactness}
Let $\Omega$ be a Lipschitz domain in $\R^n$. Assume that $K$ satisfies \eqref{Knonnegative},\eqref{Keven},\eqref{Kintegrability} and \eqref{Kboundedbelow}.
Let $\{E_n\}$ be a minimizing sequence for $P_{K,\Omega}$ and
$$\chi_{E_k}\rightarrow \chi_E \quad \mbox{in}\;\;L^1_{loc}(\R^n).$$
Then, $E$ is a minimizer for $P_{K,\Omega}$ and
$$\lim_{k\rightarrow \infty} P_{K,\Omega}(E_k)=P_{K,\Omega}(E).$$
\end{lem}
\begin{proof}
We follow the proof of Theorem 3.3 in \cite{CRS}.

Assume that $F=E$ outside $\Omega$. We set
$$F_k:=(F\cap \Omega)\cup (E_k\setminus \Omega),$$
then, by minimality of $E_k$, we have
$$ P_{K,\Omega}(F_k)\geq P_{K,\Omega}(E_k).$$
Moreover, by definition of $F_k$
$$|P_{K,\Omega}(F)-P_{K,\Omega}(F_k)|\leq L_K(\Omega,(E_k\triangle E)\setminus \Omega).$$
We denote:
$$b_k:=L_K(\Omega,(E_k\triangle E)\setminus \Omega),$$
and we get
$$P_{K,\Omega}(F)+b_k\geq P_{K,\Omega}(E_k).$$
To conclude we just need to prove that $b_k\rightarrow 0$ as $k\rightarrow \infty$, indeed, by lower semicontinuity, we would deduce that
$$P_{K,\Omega}(F)\geq \limsup_{k\rightarrow \infty}P_{K,\Omega}(E_k)\geq \liminf_{k\rightarrow \infty}P_{K,\Omega}(E_k)\geq P_{K,\Omega}(E).$$
Finally we observe that, by Remark \ref{integralK}, we have that the function
$$\phi(\bar x):=\int_{\Omega}K(x-\bar x)dx$$
belongs to $L^1(\mathcal C \Omega)$. Then, using that $\chi_{E_k}\rightarrow \chi_E$ in $L^1_{loc}$ as $k\rightarrow \infty$, the dominated convergence theorem implies
$$b_k=\int_{(E_k\triangle E)\setminus \Omega} \int_{\Omega}K(x-\bar x)dx \rightarrow 0,\quad \mbox{as}\;\;k\rightarrow \infty,$$
which concludes the proof.
\end{proof}

\section*{Appendix: Integral formulas for sets of finite perimeter}

We sketch here the
\begin{proof}[Proof of Proposition \ref{prop-fubini-typeV}]
We follow Section 5.10.2 in the book of Evans and Gariepy \cite{EvGa}.

{\em Step 1.} We show that the map $L^\perp \rightarrow \R$
\[ y \mapsto I_{E,\Omega}(L,y)  \]
is $H^{n-m}$ measurable. This follows exactly as in the proof of \cite[Lemma 1 \S 5.10.2]{EvGa} using that the supremum
in the definition of $I_{E,\Omega}(L,y)$ in  \eqref{VLy} is actually the supremum $\phi$ belonging to a countable dense subset of $C^1_c \bigl((y+L)\cap \Omega; L\cap B_1\bigr)$.

{\em Step 2.}
We prove that
\begin{equation}\label{ineq1}
\int_{L^\perp} I_{E,\Omega}(L,y) \,dH^{n-m}(y) \le  |\nabla_L u|(\Omega),
\end{equation}
where we recall that $u=\chi_E$ is a function in ${\rm BV}(\Omega)$ and
\[ |\nabla_L u|(\Omega) : = \sup \left\{ \int_{\Omega} u(x)  \,{\rm div}\, \phi(x)   \,dx \ :\ \phi\in C^1_c (\Omega;L\cap B_1)   \right\} \]
is the total variation of the projection of the (vector valued) measure $\nabla u$ onto $L$.

Let $\Omega' \subset\subset \Omega$. Define given $r>0$  define $u_r= u \ast \eta_r$ where $\eta_r = r^{-n}\left(\frac{\cdot}{r}\right)\ge 0$ is a standard smooth mollifier. Note that for $r$ small enough (depending on $\Omega'$ ) we have
\[\int_{\Omega'}  |\nabla_L u_r| \,dx \le |\nabla_L u|(\Omega),\]
where $\nabla_L$ denotes the projection of the gradient onto $L$.

Similarly as in the proof of \cite[Theorem 2 \S 5.10.2]{EvGa}, for $H^{n-m}$ a.e. $y\in L^\perp$,  we have $u_r \to u$ in $L^1$ when the two functions are restricted to the cap $\Omega\cap (y+L)$.
Hence, for $H^{n-m}$ a.e. $y$ we have
\[ I_{E,\Omega'}(L,y) \le \liminf_{r\to 0} \int_{\Omega' \cap (y+L)}  |\nabla_L u_r| \,dz. \]

Thus, Fatou's Lemma implies
\[
\begin{split}
\int_{L^\perp}  I_{E,\Omega'}(L,y)\,dy   &\le  \liminf_{r\to 0} \int_{L^\perp} \,dy \int_{\Omega' \cap (y+L)}  \,dz  \, |\nabla_L u_r|(z)
\\
&=\int_{\Omega'}  |\nabla_L u_r| \,dx \le |\nabla_L u|(\Omega).
\end{split}
\]
Then, \eqref{ineq1} follows by monotone convergence letting $\Omega'\uparrow \Omega$.

{\em Step 3.}
We prove that
\begin{equation}\label{ineq2}
  |\nabla_L u|(\Omega) \le \int_{L^\perp}  I_{E,\Omega}(L,y) \,dH^{n-m}(y).
\end{equation}

Indeed, using the definition of  $ I_{E,\Omega}(L,y)$ we find that for every given $\phi\in C^1_c (\Omega;L\cap B_1)$ we have
\[  \int_{\Omega} u(x)\,  {\rm div}\, \phi(x)   \,dx \le \int_{L^\perp} I_{E,\Omega}(L,y) \,dH^{n-m}(y).\]
Taking the supremum in $\phi$ we obtain \eqref{ineq2}.

{\em Step 4.} We show that
\begin{equation}\label{eq3}
 |\nabla_L u| (\Omega) =  \int_{\partial^*E\cap \Omega} \sqrt{ \sum_{i=1}^m \bigl({\boldsymbol v_i}\cdot \nu_E(z)\bigr)^2 } dH^{n-1}(z).
 \end{equation}

To prove \eqref{eq3} we use the divergence theorem for the set of finite perimeter $E$ and with a vector field $\phi\in C^1_c (\Omega;L\cap B_1)$.
We obtain
\begin{equation}\label{ineq4}
\begin{split}
 \int_{\Omega} u(x)  {\rm div}\,  \phi(x)   \,dx  &= \int_{\partial^*E\cap \Omega}  \phi(z) \cdot \nu_E (z) \,dH^{n-1}(z)
\\
&\le \int_{\partial^*E\cap \Omega} \sqrt{ \sum_{i=1}^m \bigl({\boldsymbol v_i}\cdot \nu_E(z)\bigr)^2 } dH^{n-1}(z).
\end{split}
\end{equation}
From this, taking supremums in the left hand side,  it easily follows that \eqref{eq3} is satisfied with the equality sign replaced by $\le$.
To prove the equality we may use the structure theorem for sets of finite perimeter to build a sequence $\phi_k$ that attain, in the limit, the equality case in \eqref{ineq4}. More precisely, this follows in a rather straightforward way from the fact that $\partial^*E$ is $H^{n-1}$ rectifiable ---see statements (i) and (ii) of Theorem 2 in Section 5.7.3 of \cite{EvGa}.

 {\em Step 4.} In the case of $m=1$ the formulas for $ I_{E,\Omega}(L,y)$ and  $ I_{E,\Omega}(L,y)_\pm$ follow by inspection using the fact that a set of finite perimeter in dimension one is (up to negligible sets) a finite union of disjoint closed intervals.
\end{proof}

\section*{Acknowledgements}

It is a pleasure to thank Xavier Cabr\'e, Francesco Maggi, and Albert Mas for interesting
conversations on the content of this paper.

This work was supported by ERC grant 277749
``E.P.S.I.L.O.N. Elliptic PDE's and Symmetry of
Interfaces and Layers for Odd Nonlinearities" and
PRIN grant
201274FYK7 ``Critical Point Theory
and Perturbative Methods for Nonlinear Differential Equations''.

\end{document}